\documentclass[12pt,reqno]{amsart}
\usepackage{amsmath, amsthm, amssymb}
\usepackage{mathtools}

\usepackage[margin=1in]{geometry}
\usepackage{parskip}
\usepackage[shortlabels]{enumitem}

\usepackage{color}

\usepackage[bookmarks,colorlinks,breaklinks]{hyperref}
\hypersetup{linkcolor=red,citecolor=blue,
  filecolor=dullmagenta,urlcolor=darkblue}

\usepackage[numbers]{natbib}
\usepackage{cleveref}

\usepackage{graphicx}
\usepackage{tikz}
\usetikzlibrary{matrix, arrows}
\usetikzlibrary{cd}

\makeatletter
\def\thm@space@setup{%
  \thm@preskip=0.5em\thm@postskip=\thm@preskip%
}
\makeatother

\newtheoremstyle{named}{}{}{\\itshape}{}{\bfseries}{.}{.5em}{\thmnote{#3's }#1}
\theoremstyle{named}

\theoremstyle{plain}
\newtheorem{thm}{Theorem}[section]
\newtheorem{prop}[thm]{Proposition}
\newtheorem{lem}[thm]{Lemma}
\newtheorem{cor}[thm]{Corollary}
\theoremstyle{definition}
\newtheorem{defn}[thm]{Definition}

\newtheorem{exmpl}[thm]{Example}

\theoremstyle{remark}
\newtheorem{rmk}[thm]{Remark}
\crefformat{footnote}{#2\footnotemark[#1]#3}

\newcommand{\gl}[2]{\mathrm{GL}_{#1}({#2})}
\renewcommand{\sl}[2]{\mathrm{SL}_{#1}({#2})}
\newcommand{\so}[2]{\mathrm{SO}_{#1}({#2})}
\renewcommand{\sp}[2]{\mathrm{Sp}_{#1}({#2})}
\newcommand{\spin}[2]{\mathrm{Spin}_{#1}({#2})}

\newcommand{\brho}{\bar\rho}

\newcommand{\bkp}{\bar\kappa}

\newcommand{\Sel}[3]{H_{#1}^1({#2},{#3})}

\newcommand{\sel}[3]{h_{#1}^1({#2},{#3})}

\newcommand{\gal}[2]{\mathrm{Gal}{(#1 / #2)}}
\newcommand{\galQ}{\Gamma_{\rats}}

\newcommand{\rats}{\mathbb{Q}}
\newcommand{\reals}{\mathbb{R}}

\newcommand{\cmplx}{\mathbb{C}}
\newcommand{\ints}{\mathbb{Z}}
\newcommand{\Qp}{\rats_p}

\newcommand{\Ql}{\rats_l}
\newcommand{\Zl}{\ints_l}
\newcommand{\bQ}{\overline{\mathbb Q}}

\newcommand{\bQl}{\overline\rats_l}

\newcommand{\bp}{\begin{pmatrix}}
\newcommand{\ep}{\end{pmatrix}}
\newcommand{\mc}{\mathcal}
\newcommand{\mf}{\mathfrak}
\newcommand{\mb}{\mathbb}

\newcommand{\frob}[1]{\mathrm{Fr}_{#1}}

\newcommand{\Zmod}[1]{\mathbb{Z}/#1 \mathbb{Z}}
\newcommand{\Zmodu}[1]{\left(\mathbb{Z}/#1 \mathbb{Z}\right)^\times}

\newcommand{\rsurj}{\twoheadrightarrow}

\newcommand{\op}[1]{\operatorname{#1}}

\makeatother
\begin{document}

\title{Algebraic monodromy groups of $l$-adic representations of $\gal{\overline\rats}{\rats}$}

\begin{abstract}
In this paper we prove that for any connected reductive algebraic group $G$ and a large enough prime $l$, there are continuous homomorphisms
$$\gal{\bQ}{\rats} \to G(\bQl)$$ with Zariski-dense image, in particular we produce the first such examples for $\op{SL}_n, \op{Sp}_{2n}, \op{Spin}_n, E_6^{sc}$ and $E_7^{sc}$. To do this, we start with a mod-$l$ representation of $\gal{\bQ}{\rats}$ related to the Weyl group of $G$ and use a variation of Stefan Patrikis' generalization of a method of Ravi Ramakrishna to deform it to characteristic zero.
\end{abstract}

\author{Shiang Tang}
\address{Department of Mathematics, University of Utah, Salt Lake City, UT 84112-0090}
\email{tang@math.utah.edu}

\maketitle

\section{Introduction}

For a split connected reductive group $G$ and a prime number $l$, it is natural to study two types of continuous representations of $\Gamma_{\rats}=\gal{\overline\rats}{\rats}$: the mod $l$ representations
$$\brho:\Gamma_{\rats} \to G(\overline{\mb F}_l)$$
and the $l$-adic representations
$$\rho:\Gamma_{\rats} \to G(\bQl)$$ where we use the discrete topology for $G(\overline{\mb F}_l)$ and the $l$-adic topology for $G(\bQl)$.
Mod $l$ representations of $\Gamma_{\rats}$ are closely related to the inverse Galois problem for finite groups of Lie type, which asks for the existence of surjective homomorphisms $\brho: \Gamma_{\rats} \rsurj G(k)$ for $k$ a finite extension of $\mb F_l$. It is still wide open, even for small groups such as $\mathrm{SL}_2$. 
If we replace $\galQ$ by $\Gamma_F$ for some number field $F$, it is not hard to show that every finite group is a Galois group over \emph{some} number field, but if we insist on $\galQ$ then the problem becomes very difficult.
On the other hand, we can ask for its analogues in the $l$-adic world:

Question 1: Are there continuous homomorphisms $\rho: \Gamma_{\rats} \to G(\bQl)$ with Zariski-dense image? 

We also ask a refined question which takes geometric Galois representations (in the sense of \cite{fm}) into account:

Question 2: Are there continuous geometric Galois representations $\rho: \Gamma_{\rats} \to G(\bQl)$ with Zariski-dense image?

This paper gives a complete answer to Question 1. We shall call a reductive group $G$ an $l$-adic algebraic monodromy group, or simply an $l$-adic monodromy group for $\galQ$ if the homomorphisms in Question 1 exist, and a geometric $l$-adic monodromy group for $\galQ$ if the homomorphisms in Question 2 exist. 
We prove the following theorem which gives an almost complete answer to Question 1: 

\begin{thm}(Main Theorem)\label{main}
Let $G$ be a connected reductive algebraic group. Then there are continuous homomorphisms 
$$\rho_l: \Gamma_{\rats} \to G(\bQl)$$ with Zariski-dense image for large enough primes $l$.
\end{thm}

The key cases of our main theorem are contained the following theorem:

\begin{thm}\label{key cases}
For a simple algebraic group $G$, there are infinitely many continuous homomorphisms
$$\rho_l: \galQ \to G(\bQl)$$ with Zariski-dense image for $l$ large enough. We impose the condition $l\equiv 1(4)$ for $G=B_n^{sc}, C_n^{sc}$, and impose $l\equiv 1(3)$ for $G=E_7^{sc}$. 
\end{thm}

\begin{rmk}\label{remove cong}
Patrikis has shown that $E_7^{ad}, E_8, F_4, G_2$ and the $L$-group of an outer form of $E_6^{ad}$ are geometric $l$-adic monodromy groups for $\galQ$, so we will not discuss these cases in the proof of the above theorem.
The congruence conditions on $l$ for $G=B_n^{sc}, C_n^{sc}, E_7^{sc}$ can be removed using a theorem proved by Fakhruddin, Khare and Patrikis (\cite{fkp}). We record it in Theorem \ref{fkp,main}.
\end{rmk}

It is shown in Cornut-Ray (\cite{cr}) that for sufficiently large regular primes $p$ (i.e., a prime $p$ that does not divide the class number of $\rats(\mu_p)$) and for a simple, adjoint group $G$, there exist a continuous representation of $\galQ$ into $G(\Qp)$ with image between the pro-p and the standard Iwahori subgroups of $G$, which generalizes a theorem of Greenberg (\cite{gre16}) for $\op{GL}_n$. In particular, the image of the Galois group is Zariski-dense. Their construction is non-geometric and is very different from ours. It is unknown whether or not there are infinitely many regular primes, however.     

It is an interesting question whether (for instance) $\op{SL}_n$ can be a geometric monodromy group for $\galQ$. The following example shows that Question 2 is more subtle than Question 1, and we should not expect an answer as clean as Theorem \ref{main}: 

\begin{exmpl}\label{sl2 mono}
Assuming the Fontaine-Mazur and the Langlands conjectures (see \cite{fm} and \cite{bg}), there is no homomorphism $\rho: \galQ \to \sl{2}{\bQl}$ that is unramified almost everywhere, potentially semi-stable at $l$, and has Zariski-dense image. 
\end{exmpl}
\begin{proof}
In fact, by the Fontaine-Mazur and the Langlands conjectures, if such $\rho$ exists, then $\rho=\rho_{\pi}$ for some cuspidal automorphic representation $\pi$ on $\gl{2}{\mb A_{\rats}}$. But $\rho$ is even, i.e., $\det \rho(c)=1$, $\pi_{\infty}$ (the archimedean component of $\pi$) is a principal series representation, and $\pi$ is associated to a Maass form. Therefore, by the Fontaine-Mazur conjecture, $\rho_{\pi}$ has finite image, a contradiction.
\end{proof}

In contrast, Theorem \ref{key cases} shows in particular that $\op{SL}_2$ is an $l$-adic monodromy group for $\galQ$. 
On the other hand, $\op{SL}_2$ can be a geometric $l$-adic monodromy group for $\Gamma_F$ for \emph{some} finite extension $F/\rats$:

\begin{exmpl}\label{sl2 mod form}
Let $f$ be a non-CM new eigenform of weight 3, level $N$, with a nontrivial nebentypus character $\varepsilon$. Such $f$ exist for suitable $N$, see \cite{lmf}.
We write $E$ for the field of coefficients of $f$. Then for all $l$ and $\lambda|l$, there is a continuous representation $r_{f,\lambda}: \galQ \to \gl{2}{E_{\lambda}}$ which is unramified outside $\{v: v|Nl \}$ and $\op{Tr}(r_{f,\lambda}(\frob{p}))=a_p$ for $p$ not dividing $Nl$, with $a_p$ the $p$-th Hecke eigenvalue of $f$. We have $\det (r_{f,\lambda})=\kappa^2 \varepsilon$ where $\kappa$ is the $l$-adic cyclotomic character. By a theorem of Ribet (\cite[Theorem 2.1]{rib85}), for almost all $l$, $\bar r_{f,\lambda}(\galQ)$ contains $\sl{2}{k}$ for a subfield $k$ of $k_{\lambda}$ (the residue field of $E_{\lambda}$). It follows that $r_{f,\lambda}$ has Zariski dense image. If we let $F$ be a finite extension of $\rats$ that trivializes $\varepsilon$, then the image of $r':=\kappa^{-1}\cdot r_{f,\lambda}|_{\Gamma_F}$ lands in $\sl{2}{E_{\lambda}}$ and is Zariski-dense. 
\end{exmpl}

The classical groups $\op{GSp}_n$, $\op{GSO}_n$ are known as geometric $l$-adic monodromy groups. Recent work of Arno Kret and Sug Woo Shin (\cite{ks}) obtains $\op{GSpin}_{2n+1}$ as a geometric $l$-adic monodromy group, and in \cite{kat17}, Nick Katz constructs geometric Galois representations with monodromy group $\op{GL}_n$. On the other hand, most of the exceptional algebraic groups are known as geometric $l$-adic monodromy groups, established in the work of Dettweiler-Reiter, Zhiwei Yun and Stefan Patrikis (\cite{dr}, \cite{yun14} and \cite{pat16}). In \cite{pat16}, Patrikis constructs geometric Galois representations for $\galQ$ with full algebraic monodromy groups for essentially all exceptional groups of adjoint type. Along the way, Patrikis has obtained an extension to general reductive groups of Ravi Ramakrishna's techniques for lifting odd two-dimensional Galois representations to geometric $l$-adic representations in \cite{ram02}. 

For the rest of this section, we sketch the strategy for proving Theorem \ref{key cases}, which makes use of Patrikis' generalization of Ramakrishna's techniques but is very different from his arguments in many ways. For the rest of this section, we assume that $G$ is a simple algebraic group defined over $\Zl$ with a split maximal torus $T$. Let $\Phi=\Phi(G,T)$ be the associated root system. Let $\mc O$ be the ring of integers of an extension of $\Ql$ whose reduction modulo its maximal ideal is isomorphic to $k$, a finite extension of $\mb F_l$. 
We start with a well-chosen mod $l$ representation and then use a variant of Ramakrishna's method to deform it to characteristic zero with big image. Achieving this is a balancing act between two difficulties:
the Inverse Galois Problem for $G(k)$ is difficult, so we want the residual image to be relatively `small'; on the other hand, Ramakrishna's method works when the residual image is `big'. 

Let us recall a construction used in \cite{pat16}. 
Patrikis uses the principal $\mathrm{GL}_2$ homomorphism to construct the residual representation
$$\brho: \Gamma_{\rats} \xrightarrow{\bar r} \gl{2}{k} \xrightarrow{\varphi} G(k)$$ for $\bar r$ a surjective homomorphism constructed from modular forms
and $\varphi$ a principal $\mathrm{GL}_2$ homomorphism (for its definition, see \cite{ser96} and \cite{pat16}, Section 7.1). 
But the principal $\mathrm{GL}_2$ is defined only when $\rho^{\vee}$ (the half-sum of coroots) is in the cocharacter lattice $X_*(T)$, which is not the case for $G=\mathrm{SL}_{2n}, \mathrm{Sp}_{2n}, E_7^{sc}$, etc.
On the other hand, the principal $\mathrm{SL}_2$ is always defined but it is not known whether $\sl{2}{k}$ is a Galois group over $\rats$ and the surjectivity of $\bar r$ is crucial in applying Ramakrishna's method.

For this reason, we use a different construction. To simplify notation, we use $G, T$ to denote $G(k), T(k)$, respectively.
We consider the following exact sequence of finite groups, which we shall refer to as the N-T sequence:
$$1 \to T \to N_G(T) \xrightarrow{\pi} \mc W \to 1$$
where $\mc W=N_G(T)/T$ is the Weyl group of $G$. We want to take $N=N_G(T)$ as the image of the residual representation $\brho$. It turns out that the adjoint action of $N$ on the Lie algebra $\mf g$ over $k$ decomposes into at most three irreducible pieces (Corollary \ref{2.2}), which is very good for applying Ramakrishna's techniques. 
It has been known for a long time that $\mc W$ is a Galois group over $\rats$, but what we need is to realize $N$ as a Galois group over $\rats$. A natural approach would be solving the embedding problem posed by the N-T sequence, i.e., to suppose there is a Galois extension $K/\rats$ realizing $\mc W$, and then to find a finite Galois extension $K'/\rats$ containing $K$ such that the natural surjective homomorphism 
$\gal{K'}{\rats} \rsurj \gal{K}{\rats}$ realizes $\pi: N \rsurj \mc W$. 

This embedding problem is solvable when the sequence \emph{splits}, by an elementary case of a famous theorem of Igor Shafarevich, see \cite{ser:tgt}, Claim 2.2.5. In \cite{ah}, the splitting of the N-T sequence is determined completely: for instance, it does not split for $G=\op{SL}_n, \op{Sp}_{2n}, \op{Spin}_n, E_7$. We find our way out by replacing $N$ with a suitable subgroup $N'$ for which the decomposition of the adjoint representation remains the same, then realizing $N'$ as a Galois group over $\rats$ with certain properties, see Section 2.1.3-2.1.6. Finally,
we define our residual representation $\brho$ to be the composite
$$\Gamma_{\rats} \rsurj N' \to G=G(k)$$ where the first arrow comes from the realization of $N'$ as a Galois group over $\rats$ and the second arrow is the inclusion map. 
We write $\brho(\mf g)$ for the Lie algebra $\mf g/k$ equipped with a $\galQ$-action induced by the homomorphism
$$\galQ \xrightarrow{\brho} G \xrightarrow{Ad} \op{GL}(\mf g).$$

Now we explain how to deform $\brho$ to characteristic zero. This is the hardest part. For a residual representation $$\brho:\galQ \to G(k)$$ unramified outside a finite set of places $S$ containing the archimedean place and a global deformation condition for $\brho$ (which consists of a local deformation condition for each $v\in S$), 
a typical question in Galois deformation theory is to find continuous $l$-adic lifts $$\rho:\galQ \to G(\mc O)$$ of $\brho$ such that for all $v$, $\rho|_{\Gamma_{\rats_v}}$ (we fix an embedding $\bQ \to \overline{\rats}_v$) satisfies the prescribed local deformation condition at $v$. If this can be done, then we can make the image of $\rho$ Zariski-dense in $G(\bQl)$ by specifying a certain type of local deformation condition at a suitable unramified prime.

Ravi Ramakrishna has an ingenious method (\cite{ram02}) for obtaining the desired lifts, which has been generalized and axiomatized in \cite{tay03}, \cite{cht}, \cite{pat16}, etc. 
By the Poitou-Tate exact sequence, if the dual Selmer group
$$\Sel{\mc L^{\perp}}{\Gamma_{\rats,S}}{\brho(\mf g)(1)}$$
associated to the global deformation condition for $\brho$ vanishes, then such lifts exist. Here $\mc L$ (resp. $\mc L^{\perp}$) is the Selmer system (resp. dual Selmer system) of tangent spaces (resp. annihilators of the tangent spaces under the Tate pairing) of the given global deformation condition. 
Ramakrishna discovered that if one imposes additional local deformation conditions of `Ramakrishna type' in place of the unramified ones at a finite set of well-chosen places of $\rats$ disjoint from $S$, then the new dual Selmer group will vanish.
However, this technique is very sensitive to the image of $\brho$, which has to be `big' to make things work:
if $\brho(\mf g)$ is irreducible then all is good, but finding such a $\brho$ can be very difficult. In practice, we would prefer those $\brho$ for which $\brho(\mf g)$ does not decompose too much. 
Unfortunately, the form of Ramakrishna's method in \cite{pat16} (see Theorem \ref{pat ram} and its proof for an account of this) does not work for our $\brho$. 

Inspired by the use of Ramakrishna's method in \cite{cht}, we surmount this by making two observations. For our $\brho$, $\brho(\mf g)$ decomposes into $\brho(\mf t)$ (the Lie algebra of $T$ over $k$ equipped with an irreducible action of $\brho(\galQ)$) and a complement (see Corollary \ref{2.2}). Our \emph{first observation} is that
if $$\Sel{\mc L}{\Gamma_{\rats,S}}{\brho(\mf t)}$$ (see Definition \ref{M Sel}) vanishes, 
then we can kill the full dual Selmer group using Ramakrishna's method; moreover, we \emph{cannot} find an auxiliary prime $w \notin S$ at which the Ramakrishna deformation condition (see Definition \ref{3.1}) satisfies $\sel{\mc L^{\perp}\cup L_w^{Ram,\perp}}{\Gamma_{\rats, S \cup w}}{\brho(\mf g)(1)}<\sel{\mc L^{\perp}}{\Gamma_{\rats,S}}{\brho(\mf g)(1)}$ when 
$\Sel{\mc L}{\Gamma_{\rats,S}}{\brho(\mf t)} \neq \{0\}$.
But it is hard to achieve $\Sel{\mc L}{\Gamma_{\rats,S}}{\brho(\mf t)}=\{0\}$ in general. 

Our \emph{second observation} is as follows: suppose that $0 \neq \sel{\mc L}{\Gamma_{\rats,S}}{\brho(\mf t)} \leq \sel{\mc L^{\perp}}{\Gamma_{\rats,S}}{\brho(\mf t)(1)}$ (the inequality is easy to guarantee), and let $\phi$ be a non-trivial class in $\Sel{\mc L^{\perp}}{\Gamma_{\rats,S}}{\brho(\mf t)(1)}$. We can then find an auxiliary prime $w \notin S$ with a Ramakrishna deformation $L_w^{Ram}$ such that $\phi|_{\Gamma_{\rats_w}} \notin L_w^{Ram,\perp}$, which implies $$\sel{\mc L^{\perp}\cup L_w^{Ram,\perp}}{\Gamma_{\rats, S \cup w}}{\brho(\mf t)(1)} < \sel{\mc L^{\perp}\cup L_w^{\perp}}{\Gamma_{\rats, S \cup w}}{\brho(\mf t)(1)},$$ where $L_w$ is the intersection of $L_w^{Ram}$ and the unramified condition at $w$. 
It turns out that (see the proof of Proposition \ref{kill tSel}) the right side of the inequality equals $\sel{\mc L^{\perp}}{\Gamma_{\rats, S}}{\brho(\mf t)(1)}$; then a double invocation of Wiles' formula gives
$$\sel{\mc L \cup L_w^{Ram}}{\Gamma_{\rats,S\cup w}}{\brho(\mf t)} < \sel{\mc L}{\Gamma_{\rats,S}}{\brho(\mf t)}.$$
By induction, we can enlarge $\mc L$ finitely many times to make $\Sel{\mc L}{\Gamma_{\rats,S}}{\brho(\mf t)}$ vanish, which in turn allows us (see the proof of Theorem \ref{lifting Weyl}) to enlarge $\mc L$ even further to make $\Sel{\mc L^{\perp}}{\Gamma_{\rats,S}}{\brho(\mf g)(1)}$ vanish, as remarked in the first observation.
Thus we obtain an $l$-adic lift $\rho: \galQ \to G(\mc O)$ satisfying the prescribed local deformation conditions. 

The above variation of Ramakrishna's method is devised very specifically for the residual representation we construct. We do not know how to generalize these ideas to the case of an arbitrary residual representation.

\begin{rmk}
For technical reasons, our method does not work for $\op{SL}_2, \op{SL}_3, \op{Spin_7}$; see Proposition \ref{2.9} and Remark \ref{2.14}. Nevertheless, an easy variant of Ramakrishna's original method applies to $\op{SL}_2$, Patrikis' extension of Ramakrishna's method applies to $\op{SL}_3$ and $\op{Spin_7}$. Patrikis' method also applies to $E^{sc}_6$ with minor modifications, and we use it in this paper. Our method should work for $E_6^{sc}$ as well, modulo an instance of the inverse Galois theory, but we do not pursue it here. 
\end{rmk}

\textbf{Acknowledgement}: 
I am extremely grateful for the patience, generosity, and support of my advisor Stefan Patrikis,
without whom this work would be impossible. 
I would also like to thank Gordan Savin for his patience and generous support on a few projects I worked on earlier in my graduate career.  
I thank Jeffrey Adams for answering my questions on his recent joint work with Xuhua He. The University of Utah mathematics department was a supportive and intellectually enriching environment, and I am grateful to all of the faculty members, postdocs, and graduate students who made it so. I would also like to thank the anonymous referees for correcting the mistakes and making many useful comments on an earlier draft of this paper.

\textbf{Notation}: For a field $F$ (typically $\rats$ or $\Qp$), we let $\Gamma_F$ denote $\gal{\overline F}{F}$ for some fixed choice of algebraic closure of $\overline F$ of $F$. When $F$ is a number field, for each place $v$ of $F$ we fix once and for all embeddings $\overline F \to \overline F_v$, giving rise to inclusions $\Gamma_{F_v} \to \Gamma_{F}$. If $S$ is a finite set of places of $F$. we let $\Gamma_{F,S}$ denote $\gal{F_S}{F}$, where $F_S$ is the maximal extension of $F$ in $\overline F$ unramified outside of $S$. If $v$ is a place of $F$ outside $S$, we write $\frob{v}$ for the corresponding arithmetic frobenius element in $\Gamma_{F,S}$. When $F=\rats$, we will sometimes write $\Gamma_v$ for $\Gamma_{F_v}$ and $\Gamma_S$ for $\Gamma_{\rats,S}$.
For a representation $\rho$ of $\Gamma_F$, we let $F(\rho)$ denote the fixed field of $\op{Ker}(\rho)$.

Consider a group $\Gamma$, a ring $A$, an algebraic group $G$ over $\op{Spec}(A)$, and a homomorphism $\rho: \Gamma \to G(A)$. We write $\mf g$ for both the Lie algebra of $G$ and the $A[G]$-module induced by the adjoint action. We let $\rho(\mf g)$ denote the $A[\Gamma]$-module with underlying $A$-module $\mf g$ induced by $\rho$. Similarly, for a $A[G]$-submodule $M$ of $\mf g$, we write $\rho(M)$ for the $A[\Gamma]$-module with underlying $A$-module $M$ induced by $\rho$.

We call an algebraic group \emph{simple} if it is connected, nonabelian and has no proper normal
algebraic subgroups except for finite subgroups. It is sometimes called an almost simple group in the literature. 
Consider a simple algebraic group $G$, we write $G^{sc}$ (resp. $G^{ad}$) for the simply-connected form (resp. adjoint form) of $G$.

Let $\mc O$ be the ring of integers of a finite extension of $\Ql$. We let $\op{CNL}_{\mc O}$ denote the category of complete noetherian local $\mc O$-algebras for which the structure map $\mc O \to R$ induces an isomorphism on residue fields.

All the Galois cohomology groups we consider will be $k$-vector spaces for $k$ a finite extension of $\mb F_l$. We abbreviate $\dim H^n(-)$ by $h^n(-)$.

We write $\kappa$ for the $l$-adic cyclotomic character, and $\bkp$ for its mod $l$ reduction.

\section{Constructions of residual representations}
In this section, we construct residual representations
$$\brho: \galQ \to G(\overline{\mathbb F}_l)$$ for $G$ a simple, \emph{simply-connected} algebraic group.

\subsection{Constructions based on the Weyl groups}
Let $k$ be a finite extension of $\mathbb F_l$. We consider the group of $k$-points of the normalizer of a split maximal torus of $G$ and hope to realize it as the Galois group of some extension of $\rats$. 

\subsubsection{Some group-theoretic results}

We recall a property of the Weyl group of an \emph{irreducible} root system $\Phi$:

\begin{lem}\label{2.1}
The Weyl group $\mc W$ acts irreducibly on the $\cmplx$-vector space spanned by $\Phi$ and transitively on roots of the same length.
\end{lem}

Let $G=G(k)$ and $T=T(k)$, a maximal split torus of $G$. 
Let $\Phi=\Phi(G,T)$ and $N=N_G(T)$.

\begin{cor}\label{2.2}
For any $\alpha,\beta\in \Phi$ of the same length, there exists $w \in N$ such that
$Ad(w)\mf g_{\alpha}=\mf g_{\beta}$. 
The adjoint action $Ad(N)$ on $\mf g$ decomposes into submodules $\mf t$ and
$$\mf g_{\Phi}:=\sum_{\alpha\in\Phi}\mf g_{\alpha}$$ when $\Phi$ is simply-laced; and is the direct sum of $\mf t$,
$$\mf g_l:=\sum_{\alpha\in\Phi, \:\textrm{$\alpha$ is long}}\mf g_{\alpha},$$
and $$\mf g_s:=\sum_{\alpha\in\Phi,\: \textrm{$\alpha$ is short}}\mf g_{\alpha}$$ otherwise. 

Moreover, as an $N$-module, $\mf t$ is irreducible; and $\mf g_{\Phi},\mf g_{l},\mf g_{s}$ are irreducible if $l$ is sufficiently large.
\end{cor}
\begin{proof}
It suffices to show that $\mf g_{\Phi},\mf g_{l},\mf g_{s}$ are irreducible $N$-modules.
We will only show that $\mf g_{\Phi}$ is irreducible, for the other two cases are similar. Take a nonzero vector $X\in \mf g_{\Phi}$, write $X=\sum_{1\leq i \leq k}X_{i}$ where $0\neq X_i \in \mf g_{\alpha_i}$ for some distinct roots $\alpha_1,\alpha_2,\cdots,\alpha_k \in \Phi$. Since $l$ is sufficiently large, we can choose $t\in T$ such that $\alpha(t)$, $\alpha\in\Phi$ are all distinct. We have
$$ad(t^j)X=\sum_i\alpha_i(t)^jX_i$$ where $0\leq j \leq k-1$. As the $X_i$'s are linearly independent and the determinant of the coefficient matrix is nonzero,
$$X,ad(t)X,ad(t^2)X,\cdots,ad(t^{k-1})X$$ are linearly independent and hence they span the same subspace as $X_1,\cdots,X_k$ do. In particular, $X_i\in \mf g_{\alpha_i}$ belongs to the $N$-submodule of $\mf g_{\Phi}$ generated by $X$. Because $N$ acts transitively on the set of root spaces, it follows that $X$ generates $\mf g_{\Phi}$. Therefore, $\mf g_{\Phi}$ is irreducible.   
\end{proof}

\begin{rmk}\label{2.3}
Corollary \ref{2.2} remains valid for a subgroup $N'$ of $N$ that maps onto a subgroup $W'$ of $W$ acting transitively on roots of the same length.
\end{rmk}

\subsubsection{Some results in the inverse Galois theory}
We record some elementary results (with proofs) about inverse Galois theory, some of which are modified in order to satisfy our purposes. See Serre's lecture notes \cite{ser:tgt} for details.

\begin{thm}\label{2.4}
For $n \geq 2$, there are infinitely many polynomials with $\rats$-coefficients that realize the symmetric group of $n$ letters $S_n$ (resp. the alternating group of $n$ letters $A_n$) as a Galois group over $\rats$. Moreover, for $S_n$ (resp. for $A_n$ with $n \geq 4$), the polynomial can be chosen to have at least a pair of non-real roots.
\end{thm}
\begin{proof}
See \cite{ser:tgt}, Section 4.4 and 4.5. For the last part, we consider the polynomial $f(X,T)$ on page 42 of \emph{loc.cit}. For any rational value of $T$, it has at most three real roots by inspection, so it must have at least a non-real root when $n \geq 4$. For $n=2,3$, it is easy to find such polynomials. The demonstration for $A_n$ is similar, see the polynomial $h(X,T)$ on page 44.
\end{proof}

The next result is an elementary case of a theorem of Igor Shafarevich, which will be used frequently in our constructions:

\begin{thm}\label{2.5}
Let $G$ be a finite group. Suppose that there is a finite Galois extension $K/\rats$ such that $\mathrm{Gal}(K/\rats) \cong G$. Let $H$ be a finite abelian group with exponent $m$. Suppose that there is a \textbf{split} exact sequence of finite groups
$$1\to H \to S \to G \to 1.$$
Then there is a finite Galois extension $M/\rats$ containing $K$ such that the natural surjective homomorphism $\mathrm{Gal}(M/\rats) \rsurj \mathrm{Gal}(K/\rats)$ realizes the surjective homomorphism $S \rsurj G$. In other words, a split embedding problem with abelian kernel always has a proper solution.

Moreover, for any prime $l$ that is outside the ramification locus of $K/\rats$ and prime to $m$, we can choose $M$ so that $l$ is unramified in $M$.
\end{thm}

\begin{proof}
The argument is a minor modification of the proof of \cite{ser:tgt}, Claim 2.2.5. Put $L=K(\mu_m)$.
$H$ can be regarded as a $m$-torsion module on which $\gal{L}{\rats}$ acts. So there is a finite free $(\Zmod{m})[\gal{L}{\rats}]$-module $F$ of which $H$ is a quotient. Suppose that $r$ is the number of copies of $(\Zmod{m})[\gal{L}{\rats}]$ in $F$.
Let $S'$ be the semi-direct product of $\gal{L}{\rats}$ and $F$. To solve the embedding problem posed by $1\to H \to S \to G \to 1$, it suffices to solve the embedding problem posed by 
$$1\to F \to S' \to \mathrm{Gal}(L/\rats) \to 1.$$
\textbf{Claim}: There is a Galois extension $M'/\rats$ that solves the above embedding problem. Moreover, for any prime $l$ that does not divide $m$ and is outside the ramification locus of $K/\rats$, we can choose $M'$ so that $l$ is unramified in $M'$.

To see this, we choose places $v_1,\cdots,v_r$ of $\rats$ away from $l$ such that $v_i$ \emph{splits completely} in $L$. Let $w_i$ be a place of $L$ extending $v_i$, $1\leq i \leq r$. Any place of $L$ extending $v_i$ can be written \emph{uniquely} as $\sigma w_i$ for some $\sigma \in \gal{L}{\rats}$. Let $w_0$ be a place of $L$ extending $l$. For $1\leq i \leq r$, choose $\theta_i\in \mc O_L$ such that for $0\leq j \leq r$, $(\sigma w_j)(\theta_i)=1$ if $\sigma=1$ and $i=j$ and $0$ if not. The existence of $\theta_i$ follows from the Weak Approximation Theorem.
Let 
$$M'=L(\sqrt[m]{\sigma\theta_i}|\sigma\in\gal{L}{\rats},1\leq i\leq r),$$ 
which is Galois over $\rats$, being the composite of $L$ and the splitting field of the polynomial $\prod_i\prod_{\sigma}(T^m-\sigma\theta_i) \in \rats[T]$. It is easy to see that
$\gal{M'}{L}$ is isomorphic to $F$ as $\gal{L}{\rats}$-modules. In fact, for each $i$ there is an isomorphism 
$$\phi_i:\mathrm{Gal}\big(L(\sqrt[m]{\sigma\theta_i}|\sigma\in\gal{L}{\rats})/L\big) \cong (\Zmod{m})[\mathrm{Gal}(L/\rats)].$$ 
For an element $g$ on the left side and any $\sigma\in \gal{L}{\rats}$, $g(\sqrt[m]{\sigma\theta_i})=\zeta_{\sigma} \cdot \sqrt[m]{\sigma\theta_i}$ for some $\zeta_{\sigma}\in \mu_m \cong \Zmod{m}$. 
We then define $$\phi_i(g)=\sum_{\sigma}\zeta_{\sigma}\cdot \sigma \in (\Zmod{m})[\gal{L}{\rats}].$$ 
It is clear that $\phi_i$ is an isomorphism by our choice of $\theta_i$. It follows that $\gal{M'}{L} \cong F$ by linear disjointness.   

Therefore, we obtain an exact sequence
$$1\to F \to \mathrm{Gal}(M'/\rats) \to \mathrm{Gal}(L/\rats) \to 1.$$
Since $F \cong \op{Ind}_{\{1\}}^{\gal{L}{\rats}}\Zmod{m}$, by Shapiro's lemma, $H^2(\gal{L}{\rats},F)=H^2(\{1\},\Zmod{m})=\{0\}$, hence the sequence splits. Thus, $\mathrm{Gal}(M'/\rats)\cong S'$.

It remains to show that $l$ is unramified in $M'$. For any $\sigma$ in $\gal{L}{\rats}$ and for any $i$, $w_0$ is unramified in $L(\sqrt[m]{\sigma\theta_i})$, because $w_0(\sigma\theta_i)=0$ and $l$ does not divide $m$. So $w_0$ is unramified in their composite $M'$. On the other hand, $l$ is unramified in $K$ by assumption and is unramified in $\rats(\mu_m)$ since $l$ does not divide $m$, so $l$ is unramified in $L$. It follows that $l$ is unramified in $M'$, proves the claim.

Finally, letting $M$ be the fixed field of the kernel of the natural surjective homomorphism $S' \rsurj S$, we obtain a solution to the original embedding problem.
\end{proof}

\subsubsection{$\op{SL}_n$} 
Let $G=\sl{n}{k}$, so $\mc W \cong S_n$. By \cite{ah}, the N-T sequence splits only when $n$ is odd. We consider the subgroup $\mc W'=A_n$ of $\mc W$. Let $T$ be the maximal torus of diagonal elements in $\sl{n}{k}$ and let $\Phi$ be $\Phi(G,T)$.

\begin{lem}\label{2.6}
Suppose $n \geq 4$. Then $A_n$, as a subgroup of $\mc W$, acts transitively on $\Phi$.
\end{lem}
\begin{proof}
This follows from the fact that $A_n$ acts doubly transitively on $\{1,2,\cdots,n\}$ if and only if $n \geq 4$.
\end{proof}

Let $N'=\pi^{-1}(\mc W')$, where $\pi$ is the natural map from $N_G(T)$ to $\mc W$.
\begin{lem}\label{2.7}
The following exact sequence of finite groups splits:
$$1 \to T \to N' \to \mc W' \to 1.$$
\end{lem}
\begin{proof}
We think of $A_n$ as a subgroup of $N=N_G(T)$ by realizing it as the group of $n\times n$ \emph{even} permutation matrices. Then $A_n$ normalizes $T$ and 
$N'$ is a semi-direct product of them.
\end{proof}

Let $l$ be large enough. Since $T$ is abelian of exponent $|k|-1$ which is prime to $l$, by Theorems \ref{2.4} and \ref{2.5}, there is a surjection
$\galQ \rsurj N'$ which is unramified at $l$. Moreover, by choosing a rational polynomial with non-real roots that realizes $A_n$ as a Galois group over $\rats$, we can make the complex conjugation map to an element away from the center of $G$. 
Define $\brho$ to be the composite
$$\galQ \rsurj N' \xrightarrow{i} \sl{n}{k}.$$

\begin{rmk}\label{2.8}
It is tricker to realize $N$ as a Galois group over $\rats$. This can be reduced to realizing the `Tits group' $\mc T$ of $\op{SL}_n$ as a Galois group. $\mc T$ can be identified with the group of $n\times n$ signed permutation matrices with determinant one, which is an index two subgroup of the group of $n\times n$ signed permutation matrices. The latter is isomorphic to the Weyl group of type $B_n$, hence known to be a Galois group over $\rats$. As the N-T sequence splits if and only if $n$ is odd, by Theorem \ref{2.5}, $\mc T$ can be realized over $\rats$ for $n$ odd. When $n$ is even, this problem is open except for small $n$, as far as the author knows.     
\end{rmk}

Let $\mf g=\mf{sl}_n(k)$.
\begin{prop}\label{2.9} For $l$ sufficiently large and $n \geq 4$, $\brho(\mf g)$ decomposes into irreducible $\galQ$-modules
$\brho(\mf t)$ and $\brho(\mf g_{\Phi})$.
\end{prop}
\begin{proof}
This follows from Corollary \ref{2.2} and Remark \ref{2.3}.
\end{proof}

There remains the case when $G$ is $\op{SL}_2$ or $\op{SL}_3$. For $\op{SL}_3$, see Section 2.2. For $\op{SL}_2$, see Section 4.2.3.

\subsubsection{$\op{Sp}_{2n}$}
Let $G=\sp{2n}{k}$, then $\mc W$ is isomorphic to a semi-direct product of $S_n$ and $D:=(\Zmod{2})^n$. We fix a maximal split torus $T$ in $\sp{2n}{k}$. The N-T sequence 
does not split by \cite{ah}.

\begin{lem}\label{2.10} Consider the N-T sequence for $\op{Sp}_{2n}$.
The group $S_n \subset \mc W$ has a section to $N$, whereas $D \subset \mc W$ does not have a section to $N$ but there is a subgroup $\tilde D$ of $N$ such that
$\pi(\tilde D)=D$ and $\tilde D \cong (\Zmod{4})^n$. 
Moreover, as subgroups of $N$, $S_n$ normalizes $\tilde D$ and $S_n \cap \tilde D=\{1\}$. 
\end{lem}
We let $\mc W_1$ be the subgroup of $N$ generated by $S_n$ and $\tilde D$.
\begin{proof}
Let $V=k^{2n}$ be the $2n$-dimensional vector space over $k$ endowed with a non-degenerate alternating form $(,)$. We may choose a basis $$e_1,\cdots,e_n,e_1',\cdots,e_n'$$ of $V$ such that $(e_i,e_j')=1$ if and only if $i=j$, and that $(e_i,e_j)=0$, $(e_i',e_j')=0$ for all $i,j$.
The Weyl group of $\mathrm{Sp}_{2n}$ is isomorphic to the semi-direct product of the group $S_n$, which acts by permuting $e_1,\cdots,e_n$, and the group $D:=(\ints/2\ints)^n$, which acts by $e_i \mapsto (\pm 1)_ie_i$. There is an inclusion $S_n \to \mathrm{Sp}_{2n}(k)$ (which is a section to $N \to \mc W$) given as follows: $\forall \sigma \in S_n$, $\sigma$ permutes $e_1,\cdots,e_n$ and $e_1',\cdots,e_n'$ by permuting the indices, which defines an element in $\mathrm{Sp}_{2n}(k)$. There is no such inclusion for $D$. However, we can define a $2$-group $\tilde D$ which embeds into $\mathrm{Sp}_{2n}(k)$ as follows:
for $1\leq i \leq n$, let $d_i$ be an endomorphism of $V$ such that $d_i(e_i)=-e_i'$, $d_i(e_i')=e_i$, and $d_i$ fixes all other basis vectors. It is clear that $d_i \in \mathrm{Sp}_{2n}(k)$. Let $\tilde D$ be the subgroup of $\mathrm{Sp}_{2n}(k)$ generated by $d_i$ for $1\leq i \leq n$. Then $\tilde D \cong (\ints/4\ints)^n$. It is obvious that (as subgroups of $\mathrm{Sp}_{2n}(k)$) $S_n$ normalizes $\tilde D$ and $S_n\cap \tilde D=\{1\}$.
\end{proof}

Therefore, by Theorem \ref{2.4} and \ref{2.5}, $\mc W_1$ can be realized as a Galois group over $\rats$ such that the complex conjugation corresponds to an element away from the center of $G$. The group $N$ is generated by $\mc W_1$ and $T$. We have $\mc W_1$ normalizes $T$ and $\mc W_1 \cap T \cong (\Zmod{2})^n$. Let $S$ be the (abstract) semidirect product of $\mc W_1$ and $T$. By Theorem \ref{2.5}, for $l$ large enough, $S$ can be realized as a Galois group over $\rats$ unramified at $l$. 
Composing the corresponding map $\galQ \rsurj S$ with the natural surjection $S \rsurj N$, we obtain a surjection
$\galQ \rsurj N$ that is unramified at $l$ and for which the complex conjugation maps to an element outside the center of $G$.
Define $\brho$ to be the composite
$$\galQ \rsurj N \xrightarrow{i} \sp{2n}{k}.$$

Let $\mf g=\mf{sp}_{2n}(k)$. The root system $\Phi$ of $\mf g$ is not simply laced.
\begin{prop}\label{2.11} 
For $l$ sufficiently large, $\brho(\mf g)$ decomposes into irreducible $\galQ$-modules $\brho(\mf t)$, $\brho(\mf g_l)$ and $\brho(\mf g_s)$.
\end{prop}
\begin{proof}
This follows from Corollary \ref{2.2}.
\end{proof}

\subsubsection{$\op{Spin}_{2n}$ and $\op{Spin}_{2n+1}$}
For spin groups, the N-T sequence 
does not split by \cite{ah}.
For $G=\op{Spin}_{2n}$, $\mc W$ is isomorphic to a semi-direct product of $S_n$ and $D:=(\Zmod{2})^{n-1}$ and we let $\mc W'$ be the subgroup generated by $A_n$ and $D$.
For $G=\op{Spin}_{2n+1}$, $\mc W$ is isomorphic to a semi-direct product of $S_n$ and $D:=(\Zmod{2})^{n}$ and we let $\mc W'$ be the subgroup generated by $A_n$ and $D$.
Similar to the symplectic case, we will show that $N'=\pi^{-1}(\mc W')$ is a Galois group over $\rats$.

\begin{lem}\label{2.12} Consider the N-T sequence for $G=\spin{2n}{k}$ or $\spin{2n+1}{k}$.
The map $\pi^{-1}(A_n) \to A_n$ admits a section, and there is a nilpotent subgroup $\tilde D$ of $N$ such that
$\pi(\tilde D)=D$. 
Moreover, as subgroups of $N$, $A_n$ normalizes $\tilde D$ and we let $\mc W_1$ be their product.
\end{lem}
\begin{proof}
Let $\underline G:=\so{2n}{k}$ or $\so{2n+1}{k}$. 
We have the standard homomorphism $i:\gl{n}{k} \to \underline G$, which restricts to a homomorphism $S_n \to \underline G$. Let $\widetilde{\op{GL}}_n(k)$ be the pullback of $i$ along the covering map $G \to \underline G$. It is a two-fold central extension of $\gl{n}{k}$, which can be identified with the group of pairs $(g,z)$ with $g \in \gl{n}{k}, z \in k^{\times}$, such that $\det g=z^2$, where the multiplication is defined by $(g_1,z_1)\cdot (g_2,z_2)=(g_1g_2,z_1z_2)$. A subgroup $H$ of $\gl{n}{k}$ has a section in $\widetilde{\op{GL}}_n(k)$ if and only if the restriction of $\det$ to $H$ is the square of a character of $H$.
In particular, taking $H=A_n$, we see that $A_n$ has a section to
$\widetilde{\op{GL}}_n(k)$. It follows that $\pi^{-1}(A_n) \to A_n$ admits a section.

The map $\underline{\pi}^{-1}(D) \to D$ has a section, where $\underline{\pi}$ is the natural map $N_{\underline G}(\underline T) \to \mc W$: This follows from \cite[Theorem 4.16]{ah}, or can be seen directly from an elementary matrix calculation. Let $\tilde D \subset G$ be the preimage of $D \subset \underline G$ under the covering map $G \to \underline G$. As $D$ is abelian, $[\tilde D, \tilde D]=Z(G) \cong \mu_2$. In particular, $\tilde D$ is nilpotent. 

Finally, because $A_n$ normalizes $D$ in $\underline G$, $A_n$ normalizes $\tilde D$ in $G$.
\end{proof}

Therefore, by Theorem \ref{2.4} and Shafarevich's Theorem (the group $\tilde D$ is nilpotent, see \cite[IX, Section 6]{nsw} for Shafarevich's Theorem), $\mc W_1$ can be realized as a Galois group over $\rats$ such that the complex conjugation corresponds to an element outside the center of $G$. Let $N'$ be the subgroup of $G$ generated by $\mc W_1$ and $T$ ($\mc W_1$ normalizes $T$). Let $S$ be the (abstract) semidirect product of $\mc W_1$ and $T$. By Theorem \ref{2.5}, for $l$ large enough, $S$ can be realized as a Galois group over $\rats$ unramified at $l$. Composing the corresponding map $\galQ \rsurj S$ with the natural quotient map $S \rsurj N'$, we obtain a surjection
$$\galQ \rsurj N'$$ that is unramified at $l$ and for which the complex conjugation maps to an element outside the center of $G$ (since the complex conjugation corresponds to an element outside the center of $G$ in the realization of $\mc W_1$ as a Galois group over $\rats$).
For $G=\op{Spin}_{2n}$ or $\op{Spin}_{2n+1}$, define $\brho$ to be the composite
$$\galQ \rsurj N' \xrightarrow{i} G(k).$$

Let $\mf g$ be the Lie algebra of $G(k)$. The corresponding root system $\Phi$ is simply laced if $G=\op{Spin}_{2n}$ and is not if $G=\op{Spin}_{2n+1}$.

\begin{prop}\label{2.13}
For $l$ sufficiently large and $n \geq 4$, $\brho(\mf g)$ decomposes into irreducible $\galQ$-modules $\brho(\mf t)$ and $\brho(\mf g_{\Phi})$ when $G=\op{Spin}_{2n}$; and it decomposes into irreducible $\galQ$-modules $\brho(\mf t)$, $\brho(\mf g_l)$ and $\brho(\mf g_s)$ when $G=\op{Spin}_{2n+1}$. 
\end{prop}
\begin{proof}
Note that the action of $\mc W'$ on $\Phi$ is transitive if and only if $n \geq 4$. Then the proposition follows from Corollary \ref{2.2} and Remark \ref{2.3}.  
\end{proof}

\begin{rmk}\label{2.14}
There remains the case when $G$ is one of $\mathrm{Spin}_4$, $\mathrm{Spin}_5$, $\mathrm{Spin}_6$, $\mathrm{Spin}_7$. But $\spin{4}{\bQl} \cong \sl{2}{\bQl} \times \sl{2}{\bQl}$, $\spin{5}{\bQl} \cong \sp{4}{\bQl}$, 
$\spin{6}{\bQl} \cong \sl{4}{\bQl}$ which are included in other cases. For $\mathrm{Spin}_7$, the half sum of coroots $\rho^{\vee}=3\alpha_1^{\vee}+5\alpha_2^{\vee}+3\alpha_3^{\vee}$ has integer coefficients, so the principal $\op{GL}_2$ map is well-defined, see Section 2.2.
\end{rmk}

\subsubsection{$E_7^{sc}$}
Let $G=E_7^{sc}(k)$. The Weyl group $\mc W$ is isomorphic to the direct product of $[\mc W,\mc W]$ and $\Zmod{2}$. By \cite{ah}, the N-T sequence does not split. We choose a subgroup $\mc W'$ of $\mc W$ which lifts to $N$ as follows. Consider the extended Dynkin diagram of type $E_7$, there is a sub-root system $\Phi'$ of $\Phi$ which is of type $A_7$. The alternating group $A_8$ is a subgroup of $S_8 \cong \mc W(A_7) \leq \mc W=\mc W(E_7)$.

\begin{lem}\label{e7,1}
The group $A_8 \leq \mc W$ lifts to $N$.
\end{lem}
\begin{proof}
This is because $A_8$ lifts to $\op{SL}_8$.
\end{proof}

\begin{lem}\label{e7,2}
The action of $A_8$ on $\Phi$ has an orbit of size 56 and an orbit of size 70.
\end{lem}
\begin{proof}
We first consider the action of $S_8\cong \mc W(A_7)$ on $\Phi$. By Lemma \ref{2.1}, $S_8$ acts transitively on $\Phi'$, which has 56 roots. A straightforward calculation using Plate $E_7$ in \cite{bou} shows that for some $\alpha \in \Phi-\Phi'$, $S_8\cdot \alpha$ has exactly 70 roots (In the extended Dynkin diagram, take $\alpha$ to be the simple root that is not in $\Phi'$, then we let the group generated by the simple reflections in $\Phi'$ act on $\alpha$ and count the number of roots in the orbit). So
the stablizer of $\alpha$ in $S_8$ is isomorphic to $S_4 \times S_4 \subset S_8$. Since $56+70=126$ is the number of roots in $\Phi$, the lemma is true for $S_8$. Now we consider the alternating group $A_8$. Lemma \ref{2.6} implies that $A_8$ still acts transitively on $\Phi'$. As $(S_4 \times S_4) \cap A_8$ (the stablizer of $\alpha$ in $A_8$) has order $\frac{1}{2}|S_4 \times S_4|=288$, the orbit
$A_8\cdot \alpha$ has exactly $|A_8|/288=70$ roots.
\end{proof}

It is clear that $A_8$, considered as a subgroup of $N$, normalizes $T$ and $A_8 \cap T=\{1\}$. Let $N'$ be the subgroup of $G=E_7^{sc}(k)$ generated by $A_8$ and $T$. By Theorem \ref{2.4} and \ref{2.5}, for $l$ large enough, we can find a continuous surjection $$\galQ \rsurj N'$$ that is unramified at $l$ and for which the complex conjugation maps to an element away from the center of $G$. Define $\brho$ to be the composite
$$\galQ \rsurj N' \xrightarrow{i} E_7^{sc}(k).$$ 

Let $\mf g_a$ (resp. $\mf g_b$) be the direct sum of the root spaces corresponding to the orbit of size 56 (resp. size 70) in Lemma \ref{e7,2}.
\begin{prop}\label{e7,3}
For $l$ sufficiently large, $\brho(\mf g)$ decomposes into irreducible $\galQ$-modules $\brho(\mf t)$, $\brho(\mf g_a)$ and $\brho(\mf g_b)$.
\end{prop}
\begin{proof}
The proof is very similar to the proof of Corollary \ref{2.2}.
\end{proof}

\subsection{The principal $\op{GL}_2$ construction}

We record some facts on the principal $\op{SL}_2$ and $\op{GL}_2$. For more details, see \cite[Section 1]{ser96}. Let $G/k$ be a simple algebraic group with a Borel $B$ containing a split maximal torus $T$ with unipotent radical $U$. Let $\Phi=\Phi(G,T)$ be the root system of $G$ with the set of simple roots $\Delta$ corresponding to $B$. We fix a pinning $\{x_{\alpha}: \mb G_a \to U_{\alpha} \}$ where $U_{\alpha}$ is the root subgroup in $B$ corresponding to $\alpha$. Let 
$X_{\alpha}=dx_{\alpha}(1)$ for all $\alpha \in \Delta$ and let 
$$X=\sum_{\alpha\in\Delta} X_{\alpha},$$ which can be extended to an $\mf{sl}_2$-triple $(X,H,Y)$ where
$$H=\sum_{\alpha>0} H_{\alpha}$$ with $H_{\alpha}$ the coroot vector corresponding to $\alpha$.

When $p$ is large enough relative to $G$, there is an exponential map 
$$\exp: \op{Lie}(U) \to U$$ which is an isomorphism.

A \emph{principal $\op{SL}_2$ homomorphism} is a homomorphism 
$$\varphi: \op{SL}_2 \to G$$
such that 
$$\varphi(\bp 1&t\\0&1 \ep)=\op{exp}(tX),$$
$$\varphi(\bp t&0\\0&t^{-1} \ep)=2\rho^{\vee}(t)$$
where $\rho^{\vee}$ is the half-sum of the coroots.
$\rho^{\vee}$ is always defined when $G$ is adjoint. Suppose that 
$\rho^{\vee}: \mb G_m \to G^{ad}$ lifts to $G$ and we fix a lift which is again denoted $\rho^{\vee}$. A \emph{principal $\op{GL}_2$ homomorphism} is a homomorphism 
$$\varphi: \op{GL}_2 \to G$$ that extends a principal $\op{SL}_2$ such that
such that
$$\varphi(\bp 1&t\\0&1 \ep)=\op{exp}(tX),$$
$$\varphi(\bp t&0\\0&1 \ep)=\rho^{\vee}(t).$$

By definition, a principal $\op{GL}_2$ factors through $\op{PGL}_2$. 

By examining the list in \cite{bou}, we get
\begin{lem}\label{pr gl2 list}
For $G$ a simple algebraic group, $\rho^{\vee}: \mb G_m \to G^{ad}$ lifts to $G^{sc}$ if and only if $G$ is one of the following types:
$A_{2n}, B_{4n}, B_{4n+3}, D_{4n}, D_{4n+1}, E_6, E_8, F_4, G_2$.
\end{lem}

The operator $ad(H)$ preserves $\mf g^{X}$ (the centralizer of $X$ in $\mf g$) and 
$$\mf g^{X}= \sum_{m>0} V_{2m},$$ where $V_{2m}$ is the eigenspace of $H$ corresponding to the eigenvalue $2m$.
The following proposition is due to Kostant; see \cite{kos59}. 

\begin{prop}\label{kostant}
The dimension of $\mf g^{X}$ is equal to the rank of $\mf g$. $V_{2m}$ is nonzero if and only if $m$ is an exponent of $\mf g$. Letting $\op{GL}_2$ act on $\mf g$ via $\varphi$, there is an isomorphism of $\op{GL}_2$-representations
$$\mf g \cong \bigoplus_{m>0} \op{Sym}^{2m}(k^2) \otimes \op{det}^{-m} \otimes V_{2m}.$$
\end{prop}

Suppose that $\rho^{\vee}: \mb G_m \to G^{ad}$ lifts to $G$.
We take $f$ to be as in Example \ref{sl2 mod form}. By \cite[Theorem 2.1]{rib85}, the projective image of $\bar r_{f,\lambda}$ is either $\op{PGL}_2(k)$ or $\op{PSL}_2(k)$ for a subfield $k$ of $k_{\lambda}$. 
We then define 
$$\brho: \galQ \xrightarrow{\bar r_{f,\lambda}} \gl{2}{k} \xrightarrow{\varphi} G(k).$$
This construction works for all exceptional groups but $E_7^{sc}$ as $\rho^{\vee}: \mb G_m \to E_7^{ad}$ does not lift to $E_7^{sc}$. We will only use this construction for $G$ of type $E_6$, $A_2$ and $B_3$.

\section{Ramakrishna's method and its variants}

Given $\brho: \galQ \to G(k)$ defined in the previous section, we want to obtain an $l$-adic lift $\rho: \galQ \to G(\mc O)$ with $\mc O$ the ring of integers of a finite extension of $\Ql$ whose residue field is $k$ satisfying a given global deformation condition. Just as in \cite{pat16}, we use Ramakrishna's method to annihilate the associated dual Selmer group. The new feature is a double use of Patrikis' extension of Ramakrishna's method when the original form fails to work, see Section 3.2.

\subsection{Ramakrishna's method}

\subsubsection{Ramakrishna deformations}
We list the key points and results of Patrikis' extension of Ramakrishna's method. For proofs, see \cite{pat16}, Section 4.2. For an overview on the deformation theory of ($G$-valued) Galois representations, see \cite{pat16}, Section 3. 

We begin by defining a type of local deformation condition called Ramakrishna's condition, which will be imposed at the auxiliary primes of ramification in Ramakrishna's global argument. Let $F$ be a finite extension of $\Qp$ for $p \neq l$, and let $\brho: \Gamma_F \to G(k)$ be an \emph{unramified} homomorphism such that $\brho(\frob{F})$ is a regular semi-simple element. Let $T$ be the connected component of the centralizer of $\brho(\frob{F})$; this is a maximal $k$-torus of $G$, but we can lift it to an $\mc O$-torus uniquely up to isomorphism, which we also denote by $T$, and then we can lift the embedding over $k$ to an embedding over $\mc O$ which is unique up to $\hat G(\mc O)$-conjugation. By passing to an \'etale extension of $\mc O$, we may assume that $T$ is split.

The following definition is from \cite{pat16}.
\begin{defn}\label{3.1} 
Let $\brho, T$ be as above. For $\alpha \in \Phi(G,T)$, $\brho$ is said to be of Ramakrishna type $\alpha$ if 
$$\alpha(\brho(\frob{F}))=\bkp(\frob{F}).$$

Let $H_{\alpha}=T\cdot U_{\alpha}$ be the subgroup generated by $T$ and the root subgroup $U_{\alpha}$ corresponding to $\alpha$. 
Ramakrishna deformation is a functor
$$\op{Lift}_{\brho}^{Ram}: \op{CNL}_{\mc O} \to \mathbf{Sets}$$
such that for a complete local noetherian $\mc O$-algebra $R$, $\op{Lift}_{\brho}^{Ram}(R)$ consists of all lifts 
$$\rho: \Gamma_F \to G(R)$$ of $\brho$ such that $\rho$ is $\hat G(R)$-conjugate to a homomorphism $\Gamma_F \xrightarrow{\rho'} H_{\alpha}(R)$ with the resulting composite
$$\Gamma_F \xrightarrow{\rho'} H_{\alpha}(R) \xrightarrow{Ad} \op{GL}(\mf g_{\alpha} \otimes R)=R^{\times}$$ equal to $\kappa$.

We shall call such a $\rho$ to be of Ramakrishna type $\alpha$ as well.

We denote by $\op{Def}_{\brho}^{Ram}$ the corresponding deformation functor.
\end{defn}

The following lemma is \cite[Lemma 4.10]{pat16}.
\begin{lem}\label{3.2} 
$\op{Lift}_{\brho}^{Ram}$ is well-defined and smooth.
\end{lem}

Consider the sub-torus $T_{\alpha}=\op{Ker}(\alpha)^0$ of $T$, and denote by $\mf t_{\alpha}$ its Lie algebra. There is a canonical decomposition $\mf t_{\alpha} \oplus \mf l_{\alpha}=\mf t$ with $\mf l_{\alpha}$ the one-dimensional torus generated by the coroot $\alpha^{\vee}$. 

The next lemma (\cite{pat16}, Lemma 4.11) is crucial in the global deformation theory.
\begin{lem}\label{3.3}
Assume $\brho$ is of Ramakrishna type $\alpha$.
Let $W=\mf t_{\alpha} \oplus \mf g_{\alpha}$, and let $W^{\perp}$ be the annihilator of $W$ under the Killing form on $\mf g$. Let $L_{\brho}^{Ram}$ (resp. $L_{\brho}^{Ram, \perp}$) be the tangent space of $\op{Def}_{\brho}^{Ram}$ (resp. the annihilator of $L_{\brho}^{Ram}$ under the local duality pairing). 
Then \begin{enumerate}
\item $L_{\brho}^{Ram} \cong H^1(\Gamma_F, \brho(W))$.
\item $\dim L_{\brho}^{Ram}=h^0(\Gamma_F, \brho(\mf g))$.
\item $L_{\brho}^{Ram,\perp} \cong H^1(\Gamma_F, \brho(W^{\perp})(1))$.
\item Let $L_{\brho}^{Ram,\Box}$ (resp. $L_{\brho}^{Ram,\perp,\Box}$) be the preimage in $Z^1(\Gamma_F, \brho(\mf g))$ of $L_{\brho}^{Ram}$ (resp. $L_{\brho}^{Ram,\perp}$). Under the canonical decomposition 
$$\mf g=\bigoplus_{\gamma} \mf g_{\gamma} \oplus \mf t_{\alpha} \oplus \mf l_{\alpha},$$ 
all cocycles in $L_{\brho}^{Ram,\Box}$ (resp. $L_{\brho}^{Ram,\perp,\Box}$) have $\mf l_{\alpha}$ (resp. $\mf g_{-\alpha}$) component equal to zero.  
\end{enumerate}
\end{lem}

\subsubsection{The global argument}
In this section, we assume $G$ is semi-simple.
Let $\brho: \Gamma_{\rats,S} \to G(k)$ be a continuous homomorphism for which $h^0(\galQ,\brho(\mf g))=h^0(\galQ,\brho(\mf g)(1))=0$.
In particular, the deformation functor is representable. The following theorem is proved in \cite[Proposition 5.2]{pat16}.
For a review on global deformation theory and systems of Selmer groups, see \cite[Sections 3.2-3.3]{pat16}.

\begin{thm}\label{pat ram}
Suppose that there is a global deformation condition $\mc L=\{L_v\}_{v \in S}$ consisting of \emph{smooth} local deformation conditions for each place $v \in S$. Let $K=\rats(\brho(\mf g),\mu_l)$.
We assume the following:
\begin{enumerate}
\item $$\sum_{v\in S}(\dim L_v) \geq \sum_{v\in S}h^0(\Gamma_{\rats_v},\brho(\mf g)).$$ 
\item $H^1(\gal{K}{\rats},\brho(\mf g))$ and $H^1(\gal{K}{\rats},\brho(\mf g)(1))$ vanish.
\item Assume Item 2 holds. For any pair of non-zero Selmer classes
$\phi \in \Sel{\mc L^{\perp}}{\Gamma_{\rats,S}}{\brho(\mf g)(1)}$ and $\psi \in \Sel{\mc L}{\Gamma_{\rats,S}}{\brho(\mf g)}$, we can restrict them to $\Gamma_K$ where they become homomorphisms, which are non-zero by Item 2. Letting $K_{\phi}/K$ and $K_{\psi}/K$ be their fixed fields, we assume that $K_{\phi}$ and $K_{\psi}$ are linearly disjoint over $K$. 
\item Consider any $\phi$ and $\psi$ as in the hypothesis of Item 3. There exists an element $\sigma \in \galQ$ such that $\brho(\sigma)$ is a regular semi-simple element of $G$, the connected component of whose centralizer we denote $T$, and such that there exists a root $\alpha \in \Phi(G,T)$ satisfying
(a) $\alpha(\brho(\sigma))=\bkp(\sigma)$.

(b) $k[\psi(\Gamma_K)]$ has an element with non-zero $\mf l_{\alpha}$ component; and 

(c) $k[\phi(\Gamma_K)]$ has an element with non-zero $\mf g_{-\alpha}$ component.
\end{enumerate}
Then there exists a finite set of primes $Q$ disjoint from $S$, and a lift
$\rho: \Gamma_{\rats, S \cup Q} \to G(\mc O)$ of $\brho$ such that
$\rho$ is of type $L_v$ at all $v \in S$ and of Ramakrishna type at all $v \in Q$.
\end{thm}

\begin{proof}
We sketch the proof for the reader's convenience. By the arguments in \cite{tay03}, Lemma 1.1 (which carry without modification to other groups), it suffices to enlarge $\mc L$ to make the dual Selmer group $\Sel{\mc L^{\perp}}{\Gamma_{\rats,S}}{\brho(\mf g)(1)}$ vanish. We may assume the dual Selmer group is nontrivial and take a nonzero class $\phi$ in it. Item 1 implies by Wiles' formula (Proposition \ref{Wiles}) that $\Sel{\mc L}{\Gamma_{\rats,S}}{\brho(\mf g)}$ is nontrivial. So we can take a nonzero class $\psi$ in it. Item 3, Item 4 and Chebotarev's density theorem all together imply there exists infinitely many $w \notin S$ such that $\psi|w \notin L_{\brho|w}^{Ram}$ and $\phi|w \notin L_{\brho|w}^{Ram,\perp}$. In particular, we have
\begin{equation}\label{psi}
\psi \notin \Sel{\mc L\cup L_{\brho|w}^{Ram}}{\Gamma_{\rats,S\cup w}}{\brho(\mf g)},
\end{equation}
and 
\begin{equation}\label{phi}
\phi \notin \Sel{\mc L^{\perp}\cup L_{\brho|w}^{Ram,\perp}}{\Gamma_{\rats,S\cup w}}{\brho(\mf g)(1)}.
\end{equation}
If we can show 
\begin{equation}\label{shrink}
\Sel{\mc L^{\perp}\cup L_{\brho|w}^{Ram,\perp}}{\Gamma_{\rats,S\cup w}}{\brho(\mf g)(1)} \subset \Sel{\mc L^{\perp}}{\Gamma_{\rats,S}}{\brho(\mf g)(1)},
\end{equation}
then (\ref{phi}) will imply that (\ref{shrink}) is a strict inclusion. The key point now is that if we let $L_w^{unr}$ denote the unramified cohomology at $w$, then $L_w=L_w^{unr} \cap L_{\brho|w}^{Ram}$ is codimension one in $L_w^{unr}$, which, together with a double invocation of Wiles' formula and (\ref{psi}), implies 
$$\Sel{\mc L^{\perp}}{\Gamma_{\rats,S}}{\brho(\mf g)(1)}=\Sel{\mc L^{\perp}\cup L_{w}^{\perp}}{\Gamma_{\rats,S\cup w}}{\brho(\mf g)(1)},$$
from which (\ref{shrink}) follows.
A variation of this argument can be found in the proof of Proposition \ref{kill tSel}. Now for the new Selmer system, Item 1 still holds (Lemma \ref{3.3}, (2)). So we can apply the above argument finitely many times until the dual Selmer group of the enlarged Selmer system vanishes. 
\end{proof}

\subsection{A variant of the global argument}
In this section, we let $G$ be a simple algebraic group and let $\brho: \galQ \to G(k)$ be as in Section 2.1. Recall that $\brho(\galQ)$ is a subgroup of $N_G(T)_k$. For $l=char(k)$ large enough, $\brho(\mf g)$ decomposes into the sum of $\brho(\mf t)$ and another one or two summands depending on whether or not $\Phi(G,T)$ is simply-laced; see Propositions \ref{2.9}, \ref{2.11}, \ref{2.13}, \ref{e7,3}.
We fix a Selmer system $\mc L$.

\begin{prop}\label{item2,3}
Assume that $l$ is large enough. Then Item 2 and Item 3 in Theorem \ref{pat ram} are satisfied.
\end{prop}
\begin{proof}
For Item 2, note that $|\gal{K}{\rats}|$ divides $(l-1)|\brho(\galQ)|$, which is prime to $l$ by the construction of $\brho$. Since the coefficients field $k$ of $H^1$ has characteristic $l$, this implies the vanishing of $H^1$.

For Item 3, since $\psi: \gal{K_{\psi}}{K}\cong \psi(\Gamma_K)$ and $\phi: \gal{K_{\phi}}{K}\cong \phi(\Gamma_K)$ are $\gal{K}{\rats}$-equivariant isomorphisms, it is enough to check that the irreducible summands in $\mf g$ and $\mf g(1)$ are non-isomorphic. We check this case by case. If $G$ of type $A_n$ (resp. $D_n$), by the construction of $\brho$, the alternating group $A_{n+1}$ (resp. $A_n$) may be identified with a subgroup of $\brho(\galQ)$. We take an element $\sigma \in \galQ$ such that $\brho(\sigma) \in A_{n}$ has order 2. Since $\rats(\brho)$ (the fixed field of $\brho$) is unramified at $l$, $\rats(\brho)$ and $\rats(\mu_l)$ are linearly disjoint over $\rats$, so we may modify $\sigma$ if necessary to make $\bkp(\sigma) \neq 1$. Consider the eigenvalues of $\sigma$ on $\mf t$ and $\mf t(1)$ (here we recall that $\mf t$ is the Lie algebra of the maximal split torus $T$ of $G$ in the construction of $\brho$); the eigenvalues on $\mf t$ are $\pm 1$, whereas none of the eigenvalues on $\mf t(1)$ can be $1$ or $-1$. Thus $\mf t$ is not isomorphic to $\mf t(1)$ as Galois modules. On the other hand, since $T:=T(k) \subset \brho(\galQ)$, we can find $\tau \in \galQ$ such that 
$\brho(\tau)$ is a regular semisimple element for which $\alpha(\brho(\tau))=a$ for $\alpha \in \Delta$, where $a$ is a generator of $\Zmodu{l}$ and $\Delta$ is a fixed set of simple roots in $\Phi$. Again since $\rats(\brho)$ and $\rats(\mu_l)$ are linearly disjoint over $\rats$, we may modify $\tau$ if necessary to make $\bkp(\tau)=a$. Consider the eigenvalues of $\tau$ on $\mf g_{\Phi}$ and $\mf g_{\Phi}(1)$, those on $\mf g_{\Phi}$ are $\{\alpha(\brho(\tau)): \alpha\in \Phi\}$, whereas those on $\mf g_{\Phi}(1)$ are $\{a\cdot \alpha(\brho(\tau)): \alpha\in \Phi\}$. Note that $\alpha(\brho(\tau))=a^{\op{ht}(\alpha)}$ and the order of $a$ is $l-1$, it is clear that these two sets are different when $l$ is large enough. Thus, $\mf g_{\Phi}$ and $\mf g_{\Phi}(1)$ are non-isomorphic as Galois modules.

If $G$ is of type $B_n$, by the construction of $\brho$, the alternating group $A_n$ may be identified with a subgroup of $\brho(\galQ)$. Using the argument in the previous paragraph, we see that $\mf t$ is not isomorphic to $\mf t(1)$ as Galois modules. On the other hand, we need to show that $\brho(\mf g_l), \brho(\mf g_s), \brho(\mf g_l)(1),\brho(\mf g_s)(1)$ are pairwisely non-isomorphic as Galois modules. Just like before we can find $\tau \in \galQ$ such that 
$\brho(\tau)$ is a regular semisimple element for which $\alpha(\brho(\tau))=a$ for $\alpha \in \Delta$ and $\bkp(\tau)=a$, where $a$ is a generator of $\Zmodu{l}$ and $\Delta$ is a fixed set of simple roots in $\Phi$. Since $l$ is large enough, the sets $\{\alpha(\brho(\tau)): \alpha\in \Phi_l\}$, $\{\alpha(\brho(\tau)): \alpha\in \Phi_s\}$,
$\{a\cdot \alpha(\brho(\tau)): \alpha\in \Phi_l\}$, $\{a\cdot \alpha(\brho(\tau)): \alpha\in \Phi_s\}$ must be distinct. So $\tau$ has different eigenvalues on $\brho(\mf g_l), \brho(\mf g_s), \brho(\mf g_l)(1),\brho(\mf g_s)(1)$ and hence they are pairwisely non-isomorphic Galois modules.

The demonstrations are the same for type $C_n$ and $E_7$.

\end{proof}

Let $M$ be a finite dimensional $k$-vector space with a continuous $\galQ$-action. Define its \emph{Tate dual} to be the space $M^{\vee}=\mathrm{Hom}(M,\mu_{\infty})$ equipped with the following $\galQ$-action: 
$$(\sigma f)(m):=\sigma(f(\sigma^{-1}m)).$$

\begin{prop}\label{Tate dual} For any continuous homomorphism $\brho: \galQ \to G(k)$,
$\brho(\mf g)^{\vee}\cong \brho(\mf g)(1)$. For $\brho$ as in Section 2.1, $\brho(\mf t)^{\vee} \cong \brho(\mf t)(1)$.
\end{prop}
\begin{proof}
As $l$ is sufficiently large, the Killing form is a non-degenerate $G$-invariant symmetric bilinear form on $\mf g$, which identifies the contragredient representation $\mf g^*$ with $\mf g$, and hence identifies $\brho(\mf g)^{\vee}$ with $\brho(\mf g)(1)$ as Galois modules. 
If $\brho$ is as in Section 2.1, then the Galois action on $\mf t$ factors through $W$. It is easy to see that the standard bilinear form on $\mf t$ is non-degenerate and $W$-invariant. Just as above, we deduce that $\brho(\mf t)^{\vee} \cong \brho(\mf t)(1)$ as Galois modules.
\end{proof}

\begin{defn}\label{M Sel}
Let $\mc L=\{L_v\}_{v \in S}$ be the Selmer system corresponding to a global deformation condition for $\brho$ that is unramified outside a finite set of places $S$, and let $\mc L^{\perp}=\{L_v^{\perp}\}_{v \in S}$ be the associated dual Selmer system.
Define the \emph{$M$-Selmer group} as follows:
$$\Sel{\mc L}{\Gamma_S}{M}=\mathrm{Ker}\left(H^1(\Gamma_S, M)\to \bigoplus_{v\in S}H^1(\Gamma_v, M)/(L_v\cap H^1(\Gamma_v, M))\right),$$
and define the \emph{$M$-dual Selmer group} as follows:
$$\Sel{\mc L^{\perp}}{\Gamma_S}{M^{\vee}}=\mathrm{Ker}\left(H^1(\Gamma_S, M^{\vee})\to \bigoplus_{v\in S}H^1(\Gamma_v, M^{\vee})/(L_v^{\perp}\cap H^1(\Gamma_v, M^{\vee}))\right).$$
\end{defn}

To apply \ref{pat ram}, we need to make sure that Item 1-Item 4 in it are satisfied. By choosing an appropriate $\mc L$, we can make Item 1 hold. Item 2 and Item 3 are satisfied by Proposition \ref{item2,3}. It is tricky to deal with Item 4: the images of $\phi$ and $\psi$, which are $\galQ$-submodules of $\brho(\mf g)$, must satisfy the group-theoretic properties in (b) and (c); if we can find an element $\sigma$ as in Item 4 such that \emph{all} summands of $\brho(\mf g)$ satisfies these properties, then Item 4 will be satisfied. So achieving Item 4 crucially depends on the group-theoretic properties of submodules of $\brho(\mf g)$.   
We need to find a regular semi-simple element $\Sigma$ in $\brho(\galQ)$, the connected component of whose centralizer we denote $T'$, for which there exists a root $\alpha' \in \Phi':=\Phi(G,T')$ such that
\begin{enumerate}
\item $\alpha'(\Sigma) \in \Zmodu{l}$, 
\item Every irreducible summand of $\brho(\mf g)$ has an element with non-zero $\mf l_{\alpha'}$ component,
\item Every irreducible summand of $\brho(\mf g)$ has an element with non-zero $\mf g_{-\alpha'}$ component.
\end{enumerate}
Unfortunately, for our residual representation, there is no such element which meets all three conditions. The rest of this section will show how Item 4 of Theorem \ref{pat ram} can be met by controlling $\psi(\Gamma_K)$ for a given class $\psi$ in the Selmer group.

If we take a regular semi-simple element $\Sigma$ in $T$ (which is the fixed split maximal torus of $G(k)$ we use when constructing the residual representation), then there is no $\alpha \in \Phi$ fulfilling both (2) and (3). Instead, we look for $\Sigma \in N=N_G(T)$ for which $\pi(\Sigma)$ is a nontrivial element in $W$, where $\pi: N_G(T) \rsurj \mc W$ is the canonical quotient map. 

\begin{lem}\label{ngt,s-s}
Assume the characteristic of $k$ is large enough for $G$. Then for any $w \in \mc W$ that fixes a non-central element of $\op{Lie}(T)$, there exist a regular semisimple element $n \in N$ (regular with respect to $G$) such that $\pi(n)=w$. If $G$ is of type $A_1, A_2, B_2, C_2$, then for any $w \in \mc W$, there exist a regular semisimple element $n \in N$ (regular with respect to $G$) such that $\pi(n)=w$.
\end{lem}  
\begin{proof}
The second part follows from a straightforward calculation. We will prove the first part. We first make the following observation: if $M$ is a Levi subgroup of $G$, then (by looking at the action of simple roots outside $M$ on elements of $Z_G(M)^0$) for every $M$-regular semisimple element $t \in M$, there is a $G$-regular semisimple element of $tZ_G(M)^0$. If $w$ fixes a non-central element of $\op{Lie}(T)$, then we take $M$ to be $Z_G(\op{Lie}(T)^w)$, which is a proper Levi subgroup of $G$ since the characteristic of $k$ is large enough for $G$. By induction on the semisimple rank, there is a $M$-regular semisimple element $n'$ such that $\pi(n')=w$, so there is a $G$-regular semisimple element $n=n'z$ with $z\in Z_G(M)^0$, and hence $\pi(n)=\pi(n'z)=\pi(n')\pi(z)=\pi(n')=w$. 
\end{proof}

\begin{rmk}
The above lemma should be true without assuming $w$ fixes a non-central element of $\op{Lie}(T)$, but the author does not know how to remove this assumption. For $\op{GL}_n$, one can show that (by matrix calculations) the property holds for all $w \in S_n$ as long as the characteristic of $k$ is large enough.
\end{rmk}

Let $t'$ be a regular semi-simple element in $G$ and $t \in T$ be an element that is conjugate to $t'$. Then $t$ and $t'$ determine a unique bijection between $\Phi=\Phi(G,T)$ and $\Phi'=\Phi(G,T')$ with $T'=Z_G(t')^0$: for any $\alpha \in \Phi$, define $\alpha'\in \Phi'$ such that
$$\alpha'(h)=\alpha(g^{-1}hg)$$ for any $h \in T'$, where $g$ is an element in $G$ such that $g^{-1}t'g=t$. Since $t$ is regular semi-simple, $\alpha'$ is independent of $g$.

Recall that for $\alpha\in\Phi$, $s_{\alpha}$ is the simple reflection in the Weyl group of $\Phi$ associated to $\alpha$.

\begin{lem}\label{s:alpha}
Assume the characteristic of $k$ is not 2. For a long root $\alpha \in \Phi$, let $\Sigma$ be a regular semi-simple element in $N$ such that
$\pi(\Sigma)=s_{\alpha}$ (which exists by Lemma \ref{ngt,s-s}). We fix an element $t\in T$ that is conjugate to $\Sigma$ and let $T'=Z_G(\Sigma)^0$. The elements $\Sigma$ and $t$ determine a bijection between $\Phi$ and $\Phi'$ as above. 
\begin{itemize}
\item If $\Phi$ is of type $A_n$ ($n \geq 1$) or $D_n$ ($n\geq 4$), then the root $\alpha' \in \Phi'$ corresponding to $\alpha$ fulfills (1) and (3). For (2), $\mf g_{\Phi}$ (recall that $\mf g_{\Phi}:=\sum_{\alpha\in\Phi} \mf g_{\alpha}$) has an element with non-zero $\mf l_{\alpha'}$ component, but $\mf t$ does not.
\item If $\Phi$ is of type $B_n$, $C_n$ ($n \geq 2$) or $E_7$, then $\alpha'$ fulfills (1) and $\mf t$ has a vector with non-zero $\mf g_{-\alpha'}$ component. 
\end{itemize}
Moreover, $\mf t' \cap \mf t=\mf t'_{\alpha'} \cap \mf t=W \cap \mf t$, where $W=\mf t'_{\alpha'} \oplus \mf g_{\alpha'}$.
\end{lem}
\begin{proof}
It is clear that $\alpha'(\Sigma)=-1$, so (1) is satisfied. We need to show
\begin{itemize}
\item The space $\mf l_{\alpha}$ has nonzero $\mf g_{-\alpha'}$ component.
\item The space $\mf g_{\alpha}$ has nonzero $\mf l_{\alpha'}$ and $\mf g_{-\alpha'}$ component.
\item $\mf t' \cap \mf t=\mf t'_{\alpha'} \cap \mf t=W \cap \mf t$.
\end{itemize}
This is essentially a $\op{GL}_2$-calculation.
We may perform the calculation in the subalgebra of $\mf g$ generated by $\mf g_{\alpha}$, whose root lattice is isomorphic to $\{x_1e_1+x_2e_2| x_i \in \ints, x_1+x_2=0\}$. We take $\alpha=e_1-e_2$ with the corresponding root vector $$X_{\alpha}:=\begin{pmatrix}
0&1\\
0&0\\
\end{pmatrix}.$$

We have 
$\Sigma=\begin{pmatrix}
0&1\\1&0
\end{pmatrix}$, $P=\begin{pmatrix}
1&1\\1&-1
\end{pmatrix}$. Then
$P^{-1} \Sigma P=diag(1,-1)$. 
The first bullet follows from the identity
$$P^{-1}\cdot diag(h_1,h_2) \cdot P=-\frac{1}{2}\begin{pmatrix}
-h_1-h_2&-h_1+h_2\\ -h_1+h_2&-h_1-h_2
\end{pmatrix}.$$
The second bullet follows from the identity
$$P^{-1} \begin{pmatrix} 0&1\\ 0&0 \end{pmatrix} P=-\frac{1}{2}\begin{pmatrix}
-1&1\\ -1&1 \end{pmatrix}.$$
To show the third bullet, note that elements in $\mf t'$ are of the form $\begin{pmatrix}
h&k\\k&h
\end{pmatrix}$, elements in $\mf t'_{\alpha'}$ are of the form $\begin{pmatrix}
h&0\\0&h
\end{pmatrix}$, and elements in $\mf g_{\alpha'}$ are of the form 
$\begin{pmatrix}
x&-x\\ x&-x \end{pmatrix}$. It follows that all three intersections in the third bullet are the one dimensional $k$-vector space spanned by $\begin{pmatrix}
1&0\\0&1
\end{pmatrix}$.

\end{proof}

\begin{lem}\label{s:alpha:beta}
Suppose the characteristic of $k$ is not 2 or 3. Assume that $\Phi$ is of type $B_n, C_n$ ($n \geq 2$) or $E_7$, so $\brho(\mf g)=\brho(\mf t) \oplus \brho(\mf g_l) \oplus \brho(\mf g_s)$ for $B_n$ and $C_n$, and $\brho(\mf g)=\brho(\mf t) \oplus \brho(\mf g_a) \oplus \brho(\mf g_b)$ for $E_7$ (see Section 2.1.4-2.1.6). 
\begin{enumerate}
\item (Type $B_n$ and $C_n$) For a pair of non-perpendicular $\beta,\gamma \in \Phi$ with $\beta$ long and $\gamma$ short, let $\Sigma$ be a regular semi-simple element in $N$ such that
$\pi(\Sigma)=s_{\beta}\cdot s_{\gamma}$ (which exists by Lemma \ref{ngt,s-s}). 
We fix an element $t\in T$ that is conjugate to $\Sigma$ and let $T'=Z_G(\Sigma)^0$. The elements $\Sigma$ and $t$ determine a bijection between $\Phi$ and $\Phi'=\Phi(G,T')$.
Then for all long root $\alpha'$ in the span of $\beta'$ and $\gamma'$ (which is a sub-system of $\Phi'$ of type $C_2$), (3) is satisfied. For (2), $\mf g_l$ (resp. $\mf g_s$) has an element with non-zero $\mf l_{\alpha'}$ component, but $\mf t$ does not. The element $\alpha'(\Sigma)$ has order 4 in $\bar k^{\times}$, hence (1) is satisfied only for $l \equiv 1(4)$.

\item (Type $E_7$) For a pair of non-perpendicular $\beta,\gamma \in \Phi$ with $\mf g_{\beta} \subset \mf g_a$ and $\mf g_{\gamma} \subset \mf g_b$, let $\Sigma$ be a regular semi-simple element in $N$ such that
$\pi(\Sigma)=s_{\beta}\cdot s_{\gamma}$. 
We fix an element $t\in T$ that is conjugate to $\Sigma$ and let $T'=Z_G(\Sigma)^0$. The elements $\Sigma$ and $t$ determine a bijection between $\Phi$ and $\Phi'=\Phi(G,T')$.
Then for all roots $\alpha'$ in the span of $\beta'$ and $\gamma'$ (which is a sub-system of $\Phi'$ of type $A_2$), (3) is satisfied. For (2), $\mf g_a$ (resp. $\mf g_b$) has an element with non-zero $\mf l_{\alpha'}$ component, but $\mf t$ does not. The elements $\alpha'(\Sigma)$ has order 3 in $\bar k^{\times}$, hence (1) is satisfied only for $l \equiv 1(3)$.
\end{enumerate}

Moreover, $W \cap \mf t \subset \mf t'$, where $W=\mf t'_{\alpha'} \oplus \mf g_{\alpha'}$. 
\end{lem}

\begin{proof}
We first prove (2). Let $\alpha'$ be any root in the span of $\beta'$ and $\gamma'$.
We need to show 
\begin{itemize}
\item $\mf l_{\beta} \oplus \mf l_{\gamma}$ has nonzero $\mf g_{-\alpha'}$ component.
\item $\mf g_{\beta}$ (resp. $\mf g_{\gamma}$) has nonzero $\mf l_{\alpha'}$ component and nonzero $\mf g_{-\alpha'}$ component.
\item $W \cap \mf t \subset \mf t'$.
\end{itemize}
We may perform the calculation in the subalgebra of $\mf g$ generated by $\mf g_{\beta}$ and $\mf g_{\gamma}$, whose root lattice is isomorphic to $\{x_1e_1+x_2e_2+x_3e_3| x_i \in \ints, x_1+x_2+x_3=0\}$. We take $\beta=e_1-e_2, \gamma=e_2-e_3$ with the corresponding root vectors $$X_{\beta}:=\begin{pmatrix}
0&1&0\\
0&0&0\\
0&0&0
\end{pmatrix},$$

$$X_{\gamma}:=\begin{pmatrix}
0&0&0\\
0&0&1\\
0&0&0
\end{pmatrix}.$$
We have 
$$\Sigma=\begin{pmatrix}
0&1&0\\
1&0&0\\
0&0&1
\end{pmatrix}\begin{pmatrix}
1&0&0\\
0&0&1\\
0&1&0
\end{pmatrix}=\begin{pmatrix}
0&0&1\\
1&0&0\\
0&1&0
\end{pmatrix}.$$
Let $r$ be a (fixed) primitive 3-rd root of unity in $\bar k$, and let $$P=\begin{pmatrix}
1&1&1\\
1&r^2&r\\
1&r&r^2
\end{pmatrix},$$
then 
$$P^{-1}\Sigma P=diag(1,r,r^2),$$
which implies $\alpha'(\Sigma)$ has order 3 in $\bar k^{\times}$.
We have $P^{-1}diag(a,b,c)P$ is a nonzero scalar multiple of
$$\begin{pmatrix}
ar+br+cr&ar+b+cr^2&ar+br^2+c\\
ar+br^2+c&ar+br+cr&ar+b+cr^2\\
ar+b+cr^2&ar+br^2+c&ar+br+cr
\end{pmatrix},$$
from which the first bullet follows.
On the other hand, we have
$P^{-1}X_{\beta}P$ is a nonzero scalar multiple of 
$$\begin{pmatrix}
r&1&r^2\\
r&1&r^2\\
r&1&r^2
\end{pmatrix},$$ 
from which it follows that $\mf g_{\beta}$ has nonzero $\mf l_{\alpha'}$ component and nonzero $\mf g_{-\alpha'}$ component for any $\alpha'$. Similarly, $\mf g_{\gamma}$ has nonzero $\mf l_{\alpha'}$ component and nonzero $\mf g_{-\alpha'}$ component for any $\alpha'$. The second bullet follows.
We now show the third bullet for $\alpha'=\beta'$, the calculation is similar for other roots. We have
$$P\begin{pmatrix}
a&x&0\\
0&a&0\\
0&0&b
\end{pmatrix}P^{-1}$$
is a nonzero constant multiple of 
$$\begin{pmatrix}
2ar+br+xr&-a+b+xr^2&-ar^2+br^2+x\\
-ar^2+br^2+xr&br+xr^2&-a+b+x\\
-a+b+xr&-ar^2+br^2+xr^2&2ar+br+x
\end{pmatrix}.$$
A simple calculation shows demanding the off-diagonal entries in the above matrix to be zero will force all of $a,b,x$ to be zero. Thus $W \cap \mf t$ is trivial in the subalgebra of $\mf g$ generated by $\mf g_{\beta}$ and $\mf g_{\gamma}$. It follows that $W \cap \mf t$ is contained in $\mf t'$.

The proof of (1) is very similar to that of (2). The computation may be performed in the subalgebra of $\mf g$ generated by $\mf g_{\beta}$ and $\mf g_{\gamma}$, which is of type $C_2$. Let $\alpha'$ be a long root in the span of $\beta'$ and $\gamma'$.
We will show the three bullets above are true. The roots are 
$\{\pm(e_1-e_2), \pm (e_1+e_2), \pm 2e_1, \pm 2e_2\}$. We fix an alternating form $x_1y_4+x_2y_3-x_3y_2-x_4y_1$ on $k^2$ and let $\beta=e_1-e_2$ and $\gamma=2e_1$ with corresponding root vectors
$$X_{\beta}:=\bp 0&1&0&0\\ 0&0&0&0\\
0&0&0&1\\
0&0&0&0
\ep,$$
$$X_{\gamma}:=\bp 0&0&0&1\\
0&0&0&0\\
0&0&0&0\\
0&0&0&0
\ep.$$

We have 
$$\Sigma=
\bp 0&1&0&0\\
1&0&0&0\\
0&0&0&1\\
0&0&1&0
\ep \bp 0&0&0&1\\
0&1&0&0\\
0&0&1&0\\
-1&0&0&0
\ep
=\bp
0&1&0&0\\
0&0&0&1\\
-1&0&0&0\\
0&0&1&0
\ep.$$
Let $r$ be a (fixed) primitive 8-th root of unity in $\bar k$, and let 
$$P=\bp 1&1&1&1\\
r&r^3&r^5&r^7\\
r^3&r&r^7&r^5\\
r^2&r^6&r^2&r^6
\ep,$$
then 
$$P^{-1}\Sigma P=diag(r,r^3,r^5,r^7),$$
which implies $\alpha'(\Sigma)$ has order 4 in $\bar k^{\times}$.
We have $P^{-1}diag(a,b,-b,-a)P$ is a nonzero scalar multiple of 
$$\bp 
0&a-br^6&0&a-br^2\\ 
a-br^6&0&a+br^2&0\\
0&a+br^6&0&a+br^2\\
a-br^6&0&a-br^2&0
\ep.$$
Let $\alpha' \in \Phi'$ be any long root, then $\alpha'$ satisfies the first bullet. On the other hand,
$P^{-1}X_{\beta}P$ is a nonzero scalar multiple of
$$\bp r-r^7&0&r^5-r^7&r^7-r^3\\
0&r^3+r&2r^5&r^7+r\\
r+r^7&2r^3&r^5+r^7&0\\
r-r^5&r^3-r&0&r^7-r
\ep,$$
and $P^{-1}X_{\gamma}P$ is a nonzero scalar multiple of
$$\bp r^2&r^6&r^2&r^6\\
r^2&r^6&r^2&r^6\\
r^2&r^6&r^2&r^6\\
r^2&r^6&r^2&r^6
\ep,$$
from which the second bullet follows.
We check the third bullet for $\alpha'$ corresponding to $2e_2$, the calculation is similar for other long roots in $\Phi'$.
We have 
$$P \bp a&0&0&0\\
0&0&x&0\\
0&0&0&0\\
0&0&0&-a
\ep P^{-1}$$
is a nonzero scalar multiple of
$$\bp x & a(r^5-r^3)+xr^3 & a(r^5-r^3)-xr^5 & a(r^6-r^2)+xr^6\\
a(r-r^7)+xr^3 & r^6x & 2ar^6-x & a(r^5-r^3)+xr\\
a(r^3-r^5)+xr & 2ar^2-x & -r^6x & a(r^3-r^5)+xr^7\\
-2ar^6+xr^6 & a(r^3-r^5)+xr & a(r^7-r)-xr^3 & -x
\ep.$$
It is easy to see that for the matrix to be diagonal, both of $a$ and $x$ have to be zero. Thus $W \cap \mf t$ is trivial in the subalgebra of $\mf g$ generated by $\mf g_{\beta}$ and $\mf g_{\gamma}$. It follows that $W \cap \mf t$ is contained in $\mf t'$.
\end{proof}

Let $\Sigma, \alpha'$ be as in Lemma \ref{s:alpha}, (1) for $\Phi$ of type $A_n$ or $D_n$, and Lemma \ref{s:alpha:beta} for $\Phi$ of type $B_n, C_n$ or $E_7$. We have $\alpha(\Sigma) \in \Zmodu{l}$ for primes $l$ satisfying an appropriate congruence condition if necessary. We need to modify the element $\Sigma$ to make it land in $\brho(\galQ)$. When $\Phi$ is of type $A_n$ or $D_n$, $\pi(\brho(\galQ))=[\mc W,\mc W]$ (see 2.1.3 and 2.1.5). For any $\alpha\in\Phi$, we write $\mf{sl}^{\alpha}_2$ for the Lie subalgebra of $\mf g$ generated by $\mf g_{\alpha}$ and $\mf g_{-\alpha}$. We replace $s_{\alpha}$ by $s_{\alpha}s_{\beta}$ for some root $\beta$ orthogonal to $\alpha$ such that $[\mf{sl}^{\alpha}_2,\mf{sl}^{\beta}_2]$ is trivial (such $\beta$ exists because $n \geq 4$) and replace $\Sigma$ with a regular semisimple element in $N_G(T)$ that maps to $s_{\alpha}s_{\beta}$ when modulo $T(k)$. Note that $s_{\alpha}s_{\beta}\in [\mc W,\mc W]$: The Weyl group $\mc W$ acts transitively on the irreducible root system $\Phi$, so there exists $w \in \mc W$ such that $w\alpha=\beta$, and hence $ws_{\alpha}w^{-1}=s_{\beta}$; it follows that $s_{\alpha}s_{\beta}=[s_{\alpha},w]$. We again denote this new element by $\Sigma$, which now lands in $\brho(\galQ)$. Lemma \ref{s:alpha}, (1) still holds. When $\Phi$ is of type $B_n$, we again have $\pi(\brho(\galQ))=[\mc W,\mc W]$ (see 2.1.5). As $s_{\beta}s_{\gamma} \in [\mc W,\mc W]$, $\Sigma \in \brho(\galQ)$, so no modification is needed. When $\Phi$ is of type $C_n$, $\pi(\brho(\galQ))=\mc W$ (see 2.1.4), so we automatically have $\Sigma \in \brho(\galQ)$. When $\Phi$ is of type $E_7$, the corresponding element $\Sigma$ is in $\brho(\galQ)$ as well (see 2.1.6).   
Since $\brho$ is unramified at $l$ and $\rats(\mu_l)$ is totally ramified at $l$, $\rats(\brho)$ and $\rats(\mu_l)$ are linearly disjoint over $\rats$. So there exists an element $\sigma \in \galQ$ such that
$\brho(\sigma)=\Sigma$ and $\bkp(\sigma)=\alpha'(\Sigma)$. It follows that $\alpha'(\brho(\sigma))=\bkp(\sigma)$.

\begin{lem}\label{tSelgSel}
Suppose there is a Selmer system $\mc L=\{L_v\}_{v \in S}$ for which the $\mf t$-Selmer group
$\Sel{\mc L}{\Gamma_S}{\brho(\mf t)}$ is trivial. 
We take a pair of non-zero Selmer classes
$\phi \in \Sel{\mc L^{\perp}}{\Gamma_{\rats,S}}{\brho(\mf g)(1)}$ and $\psi \in \Sel{\mc L}{\Gamma_{\rats,S}}{\brho(\mf g)}$.
Then Item 4 of Theorem \ref{pat ram} is satisfied.
\end{lem}

\begin{proof}
We need to check Theorem \ref{pat ram}, Item 4, (b) and (c).  
First, Proposition \ref{item2,3} and the inflation-restriction sequence imply that $\psi(\Gamma_K)$ and $\phi(\Gamma_K)$ are nontrivial. 
By Lemma \ref{s:alpha} and \ref{s:alpha:beta}, every irreducible summand of $\brho(\mf g)$ has an element with non-zero $\mf g_{-\alpha'}$ component. In particular, (c) holds.
As $\Sel{\mc L}{\Gamma_S}{\brho(\mf t)}$ is trivial, $\psi(\Gamma_K) \nsubseteq \mf t$, which implies that $k[\psi(\Gamma_K)]$ contains $\mf g_{\Phi}$ when $\Phi$ is of type $A_n$ or $D_n$, $k[\psi(\Gamma_K)]$ contains one of $\mf g_{l}$ and $\mf g_{s}$ when $\Phi$ is of type $B_n$ or $C_n$, and $k[\psi(\Gamma_K)]$ contains one of $\mf g_{a}$ and $\mf g_{b}$ when $\Phi$ is of type $E_7$. It then follows from Lemma \ref{s:alpha} and \ref{s:alpha:beta} that $k[\psi(\Gamma_K)]$ has an element with nonzero $\mf l_{\alpha'}$ component. So (b) holds as well.
\end{proof}

The next proposition achieves the vanishing assumption of the $\mf t$-Selmer in Lemma \ref{tSelgSel} by using of a variant of the cohomological arguments in Ramakrishna's method.
\begin{prop}\label{kill tSel}
Suppose that
$$h_{\mc L}^1(\Gamma_S,\brho(\mf t))\leq h_{\mc L^{\perp}}^1(\Gamma_S,\brho(\mf t)(1)).$$
Then there is a finite set of places $Q$ disjoint from $S$ and a Ramakrishna deformation condition for each $w\in Q$ with tangent space $L_w^{Ram}$ such that
$$H_{\mc L \cup \{L_w^{Ram}\}_{w\in Q} }^1(\Gamma_{S\cup Q},\brho(\mf t))=0.$$
\end{prop}

We may assume that $H_{\mc L}^1(\Gamma_S,\brho(\mf t))$ is nontrivial, for otherwise we are done. The inequality in Proposition \ref{kill tSel} then implies that $H_{\mc L^{\perp}}^1(\Gamma_S,\brho(\mf t)(1))$ is nontrivial. Let $0\neq \phi \in H_{\mc L^{\perp}}^1(\Gamma_S,\brho(\mf t)(1))$.

\begin{lem}\label{3.12}
There exists $\tau\in\Gamma_{\rats}$ with the following properties: 
\begin{enumerate}
\item $\bar\rho(\tau)$ is a regular semisimple element of $G(k)$, the connected component of whose centralizer we denote $T'$.
\item There exists $\alpha' \in \Phi(G,T')$, such that $\alpha'(\bar\rho(\tau))=\bar\kappa(\tau)$. 
\item $k[\phi(\Gamma_K)]$ has an element with nonzero $\mf g_{-\alpha'}$-component.
\end{enumerate}
\end{lem}
\begin{proof}
We have seen that the groups $H^1(\gal{K}{\rats},\brho(\mf g))$ and $H^1(\gal{K}{\rats},\brho(\mf g)(1))$ are both trivial. In particular, the groups $H^1(\gal{K}{\rats},\brho(\mf t))$ and $H^1(\gal{K}{\rats},\brho(\mf t)(1))$ are both trivial. The restriction-inflation sequence then implies that $\phi(\Gamma_K)$ is nontrivial. 
Now we let $\Sigma$, $\alpha'$ be as in Lemma \ref{s:alpha}, (1) for $\Phi$ of type $A_n$ or $D_n$, and Lemma \ref{s:alpha}, (2) for $\Phi$ of type $B_n$, $C_n$ or $E_7$. If necessary, we can modify $\Sigma$ to make it land in $\brho(\galQ)$, as explained in the paragraph preceding Lemma \ref{tSelgSel}. We have $\alpha(\Sigma) \in \Zmodu{l}$.
Since $\brho$ is unramified at $l$ and $\rats(\mu_l)$ is totally ramified at $l$, $\rats(\brho)$ and $\rats(\mu_l)$ are linearly disjoint over $\rats$. So there exists an element $\tau \in \galQ$ such that
$\brho(\tau)=\Sigma$ and $\bkp(\tau)=\alpha'(\Sigma)$. It follows that $\alpha'(\brho(\tau))=\bkp(\tau)$, proves (2). Statement (3) follows from Lemma \ref{s:alpha}.
\end{proof}

\begin{cor}\label{3.13}
There exist infinitely many places $w\notin S$ such that $\bar\rho|_{\Gamma_w}$ is of Ramakrishna type $\alpha'$ and $\phi|_{\Gamma_w} \notin L_w^{Ram,\perp}$.
\end{cor}
\begin{proof}
This follows from Lemma \ref{3.3}, \ref{3.12} and Chebotarev's density theorem. See the proof of \cite{pat16}, Lemma 5.3.
\end{proof}

\textbf{Proof of Proposition \ref{kill tSel}}:
Let $w$ be chosen as in Lemma \ref{3.13}. We will show that 
\begin{equation}\label{1}
h^1_{\mc L^{\perp}\cup L_w^{Ram,\perp}}(\Gamma_{S\cup w},\brho(\mf t)(1))<h^1_{\mc L^{\perp}}(\Gamma_{S},\brho(\mf t)(1))
\end{equation}
and 
\begin{equation}\label{2}
h^1_{\mc L\cup L_w^{Ram}}(\Gamma_{S\cup w},\brho(\mf t))-h^1_{\mc L^{\perp}\cup L_w^{Ram,\perp}}(\Gamma_{S\cup w},\brho(\mf t)(1))=h^1_{\mc L}(\Gamma_{S},\brho(\mf t))-h^1_{\mc L^{\perp}}(\Gamma_{S},\brho(\mf t)(1))
\end{equation}
which imply that
$$h^1_{\mc L\cup L_w^{Ram}}(\Gamma_{S\cup w},\brho(\mf t))<h^1_{\mc L}(\Gamma_{S},\brho(\mf t)),$$
from which Proposition \ref{kill tSel} follows by induction.

We first show (\ref{2}): By a double invocation of Wiles' formula (see Proposition \ref{Wiles}), the difference between the two sides equals 
$\dim(L_w^{Ram}\cap H^1(\Gamma_w,\brho(\mf t)))-h^0(\Gamma_w,\brho(\mf t))=h^1(\Gamma_w,\brho(W\cap \mf t))-h^0(\Gamma_w,\brho(\mf t))$. As $H^0(\Gamma_w, \brho(\mf g))=\mf t'$, we have $H^0(\Gamma_w, \brho(\mf t))=\mf t \cap \mf t'$; on the other hand, the action of $\brho(\Gamma_w)$ on $W\cap \mf t$ is a sum of the trivial representation and the cyclotomic character. By an elementary calculation in Galois cohomology, 
$h^1(\Gamma_w, k)=h^1(\Gamma_w, k(1))=1$. It follows that $h^1(\Gamma_w,\brho(W\cap \mf t))=\dim W\cap \mf t$ and 
so $h^1(\Gamma_w,\brho(W\cap \mf t))-h^0(\Gamma_w,\brho(\mf t))=\dim(W\cap \mf t)-\dim(\mf t'\cap \mf t)$, which is zero by Lemma \ref{s:alpha}.

It remains to prove (\ref{1}). Let $L_w=L_w^{unr} \cap L_w^{Ram}$, so $L_w^{\perp}=L_w^{unr,\perp}+L_w^{Ram,\perp}$. We have the following obvious inclusions
\begin{equation}\label{a}
H^1_{\mc L\cup L_w}(\Gamma_{S\cup w},\brho(\mf t)) \subset H^1_{\mc L\cup L_w^{Ram}}(\Gamma_{S\cup w},\brho(\mf t)),
\end{equation}

\begin{equation}\label{a'}
H^1_{\mc L^{\perp}\cup L_w^{Ram,\perp}}(\Gamma_{S\cup w},\brho(\mf t)(1))\subset H^1_{\mc L^{\perp}\cup L_w^{\perp}}(\Gamma_{S\cup w},\brho(\mf t)(1)),
\end{equation}

\begin{equation}\label{b}
H^1_{\mc L\cup L_w}(\Gamma_{S\cup w},\brho(\mf t)) \subset H^1_{\mc L\cup L_w^{unr}}(\Gamma_{S\cup w},\brho(\mf t))=H^1_{\mc L}(\Gamma_S,\brho(\mf t)),
\end{equation}

\begin{equation}\label{b'}
H^1_{\mc L^{\perp}}(\Gamma_S,\brho(\mf t)(1))=H^1_{\mc L^{\perp}\cup L_w^{unr,\perp}}(\Gamma_{S\cup w},\brho(\mf t)(1)) \subset H^1_{\mc L^{\perp}\cup L_w^{\perp}}(\Gamma_{S\cup w},\brho(\mf t)(1)).
\end{equation}

As $\phi|_{\Gamma_w} \notin L_w^{Ram,\perp}$, (\ref{a'}) is a strict inclusion. We claim that (\ref{b'}) is an isomorphism, which will imply (\ref{1}). To prove our claim, we consider (\ref{b}) first. There is an exact sequence 
$$0 \to H^1_{\mc L\cup L_w}(\Gamma_{S\cup w},\brho(\mf t)) \to H^1_{\mc L}(\Gamma_S,\brho(\mf t)) \to (L_w^{unr}\cap H^1(\Gamma_w,\brho(\mf t)))/(L_w\cap H^1(\Gamma_w,\brho(\mf t))).$$
As $$L_w^{unr}=H^1(\Gamma_w/I_w, \brho(\mf g)) \xrightarrow{f \mapsto f(\frob{w})} \mf g/(\brho(\frob{w})-1)\mf g \cong \mf t',$$ the top of its last term is isomorphic to $H^1(\Gamma_w, \brho(\mf t'\cap \mf t))$, which
has dimension $\dim(\mf t'\cap \mf t)$; the bottom of its last term is isomorphic to $H^1(\Gamma_w, \mf t'_{\alpha}\cap \mf t)$, which has dimension $\dim(\mf t'_{\alpha}\cap \mf t)$. By Lemma \ref{s:alpha}, these dimensions are equal. So the last term is zero, and hence (\ref{b}) is an isomorphism. 
A double invocation of Wiles' formula (Proposition \ref{Wiles}) shows
$h^1_{\mc L^{\perp}\cup L_w^{\perp}}(\Gamma_{S\cup w},\brho(\mf t)(1))-h^1_{\mc L^{\perp}}(\Gamma_S,\brho(\mf t)(1))$ equals
$$h^1_{\mc L\cup L_w}(\Gamma_{S\cup w},\brho(\mf t))-h^1_{\mc L}(\Gamma_S,\brho(\mf t))+h^0(\Gamma_w,\brho(\mf t))-\dim(L_w\cap H^1(\Gamma_w,\brho(\mf t))).$$

Because (\ref{b}) is an isomorphism and $\dim(L_w\cap H^1(\Gamma_w,\brho(\mf t)))=h^1(\Gamma_w, \brho(W \cap \mf t' \cap \mf t)) =\dim(W\cap \mf t' \cap \mf t)=\dim(\mf t' \cap \mf t)=h^0(\Gamma_w,\brho(\mf t))$ (Lemma \ref{s:alpha}), the right hand side of the above identity is zero. Therefore, (\ref{b'}) is an isomorphism, which completes the proof of the proposition. \hfill $\Box$

\begin{thm}\label{lifting Weyl}
Let $\mc L=\{L_v\}_{v \in S}$ be a family of smooth local deformation conditions for $\bar\rho$ (the residual representation defined in Section 2.1) unramified outside a finite set of places $S$ containing the real place and all places where $\brho$ is ramified. Suppose that 

$$\sum_{v\in S}\dim L_v \geq \sum_{v\in S} h^0(\Gamma_v,\brho(\mf g)),$$
$$\sum_{v\in S}\dim(L_v \cap H^1(\Gamma_v,\brho(\mf t))) \leq \sum_{v\in S} h^0(\Gamma_v,\brho(\mf t)).$$

Assume $l$ is large enough; in addition, if $\Phi$ is of type $E_7$, assume $l\equiv 1(3)$, and if $\Phi$ is doubly-laced, assume $l\equiv 1(4)$.

Then there is a finite set of places $Q$ disjoint from $S$ and a continuous lift $$\rho:\Gamma_{S\cup Q} \to G(\mc O)$$ of $\bar\rho$ such that
$\rho$ is of type $L_v$ for $v\in S$ and of Ramakrishna type for 
$v\in Q$.
\end{thm}

\begin{proof}
The second inequality and Wiles' formula (Proposition \ref{Wiles}) imply that
$h_{\mc L}^1(\Gamma_S,\brho(\mf t))\leq h_{\mc L^{\perp}}^1(\Gamma_S,\brho(\mf t)(1))$. By 
Proposition \ref{kill tSel}, we can enlarge $\mc L$ by adding finitely many Ramakrishna deformation conditions to get a new Selmer system $\mc L'=\{L_v\}_{v \in S'}$ with $S' \supset S$ such that
$H_{\mc L'}^1(\Gamma_{S'},\brho(\mf t))=\{0\}$. By Lemma \ref{3.3} (2), replacing $\mc L$ by $\mc L'$ preserves the first inequality.

We choose $\sigma \in \galQ$ and $\alpha' \in \Phi(G,T')$ as in the paragraph preceding Lemma \ref{tSelgSel} (this is where the congruence conditions for $\Phi$ of type $B_n, C_n, E_7$ come in). Chebotarev's density theorem implies that there are infinitely many places $w \notin S'$ such that $\brho|_w$ is of Ramakrishna type $\alpha'$.
We have for such a prime $w$,
$$H^1_{\mc L' \cup L_w^{Ram}}(\Gamma_{S'\cup w},\brho(\mf t))=0.$$ 
In other words, adding a Ramakrishna local deformation conditions does not make the $\mf t$-Selmer group jump back to a nontrivial group.
Indeed, $L_w^{Ram}\cap H^1(\Gamma_w,\brho(\mf t))=H^1(\Gamma_w,\brho(W\cap \mf t))\subset H^1(\Gamma_w,\brho(\mf t'\cap \mf t))=H^1(\Gamma_w/I_w,\brho(\mf t))$, where the middle inclusion follows from Lemma \ref{s:alpha} and \ref{s:alpha:beta}. So 
$$H^1_{\mc L' \cup L_w^{Ram}}(\Gamma_{S'\cup w},\brho(\mf t)) \subset H_{\mc L'}^1(\Gamma_{S'},\brho(\mf t))=\{0\}.$$ 

Let us check the assumptions of Theorem \ref{pat ram}. By the first inequality in the assumption and Proposition \ref{Wiles}, Item 1 holds. Item 2 and 3 are satisfied by Proposition \ref{item2,3}.  As $H_{\mc L'}^1(\Gamma_{S'},\brho(\mf t))=\{0\}$, Item 4 is satisfied by Lemma \ref{tSelgSel}.
Therefore, by the proof of Theorem \ref{pat ram}, there is a strict inclusion
$$\Sel{\mc L'^{\perp}\cup L_{w}^{Ram,\perp}}{\Gamma_{\rats,S'\cup w}}{\brho(\mf g)(1)} \subset \Sel{\mc L'^{\perp}}{\Gamma_{\rats,S'}}{\brho(\mf g)(1)}.$$
As the $\mf t$-Selmer group is still trivial for the enlarged Selmer system, Item 4 remains valid by Lemma \ref{tSelgSel}. So we can find a prime $w' \notin S' \cup w$ and enlarge the Selmer system $\mc L' \cup L^{Ram}_w$ in the same way so that the dual Selmer group shrinks even further. Applying this argument finitely many times, we can kill the dual Selmer group. Therefore, by the first two lines of the proof of Theorem \ref{pat ram}, we obtain desired $l$-adic lifts.   
\end{proof}

\subsection{Deforming principal $\op{GL}_2$}

We use the notation in Section 2.2. Recall that $\brho$ is the composite $\galQ \to \gl{2}{k} \xrightarrow{\varphi} G(k)$ 
where the first map is constructed from modular forms and the second map is the principal $GL_2$-map.

Patrikis has shown that all simple algebraic groups of exceptional types are geometric monodromy groups for $\galQ$ except for $E_6^{ad}, E_6^{sc}, E_7^{sc}$ (\cite{pat16}). In this section, we follow Patrikis' work and use the principal $\op{GL}_2$ to construct full-image Galois representations into $E_6^{ad}, E_6^{sc}, \op{SL}_3, \op{Spin}_7$. 

The proof of the following theorem is identical to that of \cite{pat16}, Theorem 7.4.
\begin{thm}\label{3.15}
Let $\mc L=\{L_v\}_{v \in S}$ be a family of smooth local deformation conditions for $\bar\rho$ (the residual representation defined in Section 2.2) unramified outside a finite set of places $S$ containing the real place and all places where $\brho$ is ramified. Suppose that 
$$\sum_{v\in S}\dim L_v \geq \sum_{v\in S} h^0(\Gamma_v,\brho(\mf g)).$$ 
Assume $l$ is large enough.

Then there is a finite set of places $Q$ disjoint from $S$ and a continuous lift $$\rho:\Gamma_{S\cup Q} \to G(\mc O)$$ of $\bar\rho$ such that
$\rho$ is of type $L_v$ for $v\in S$ and of Ramakrishna type for 
$v\in Q$.
\end{thm}

In \cite[Lemma 7.6]{pat16}, the following fact is verified using Magma.
\begin{lem}\label{lie excep}
Assume $l$ is large enough for $\mf g$. For $\mf g$ of exceptional type, there is a root $\alpha \in \Phi$ such that every irreducible submodule of $\brho(\mf g)$ has a vector with nonzero $\mf l_{\alpha}$ component and a vector with nonzero $\mf g_{-\alpha}$ component.
\end{lem}

For our purpose, we only need to establish its analogs for $\mf g$ of type $A_n$ and $B_n$. 

\begin{lem}\label{lie An}
Assume $l$ is large enough for $\mf g$. For $\mf g$ of type $A_n$, there is a root $\alpha \in \Phi$ such that every irreducible submodule of $\brho(\mf g)$ has a vector with nonzero $\mf l_{\alpha}$ component and a vector with nonzero $\mf g_{-\alpha}$ component.
\end{lem}
\begin{proof}
Let $\mf g=\mf{sl}_{n+1}$ and let $\alpha_{i,j}=e_i-e_j$, $i \neq j$ be the roots of $\mf g$. Let $E_{i,j}$ be the $n+1$ by $n+1$ matrix that has $1$ at the $(i,j)$-entry and zeros elsewhere. The $sl_2$-triple associated to $\alpha_{i,j}$ is $\{X_{i,j}:=E_{i,j}, H_{i,j}=E_{i,i}-E_{j,j}, Y_{i,j}:=E_{j,i}\}$.
Let $$X=X_{1,2}+X_{2,3}+\cdots +X_{n,n+1},$$ 
$$H=\sum_{i<j}H_{i,j}=k_1H_{1,2}+k_2H_{2,3}+\cdots +k_nH_{n,n+1},$$
$$Y=k_1Y_{1,2}+k_2Y_{2,3}+\cdots +k_nY_{n,n+1},$$
where $k_i:=i(n-i+1)$. The triple $\{X,H,Y\}$ is an $sl_2$-triple containing the regular unipotent element $X$.

A straightforward calculation gives for $i<j$
$$[Y,X_{i,j}]=k_iX_{i+1,j}-k_{j-1}X_{i,j-1}.$$
Put $h=j-i$ and apply the above identity recursively, we obtain
$$(adY)^hX_{i,j}$$
$$=(-1)^hk_{i,j}\big( H_{i,i+1}-\binom{h-1}{1}H_{i+1,i+2}+\binom{h-1}{2}H_{i+2,i+3}+\cdots +(-1)^{h-1}\binom{h-1}{h-1}H_{j-1} \big),$$
where $k_{i,j}=k_ik_{i+1}\cdots k_{j-1}$.
By Proposition \ref{kostant},
$$\mf g^{X}=\sum_{h=1}^n <v_{2h}>$$
where $v_{2h}=\sum_{j-i=h} X_{i,j}$. Then we have
$$(adY)^hv_{2h}=h_1H_{1,2}+\cdots +h_nH_{n,n+1}$$
with $h_1=(-1)^hk_{1,h+1}$, $h_2=(-1)^{h-1}\binom{h-1}{1}k_{1,h+1}+(-1)^hk_{2,h+2}$ and $h_i=(-1)^{h-1}h_{n-i+1}$.
Since $(adY_{i,i+1})H_{i-1,i}=Y_{i,i+1}$, $(adY_{i,i+1})H_{i,i+1}=-2Y_{i,i+1}$ and $(adY_{i,i+1})H_{i+1,i+2}=Y_{i,i+1}$, 
$$(adY)^{h+1}v_{2h}=(adY)(h_1H_{1,2}+\cdots +h_nH_{n,n+1})$$
$$=k_1(-2h_1+h_2)Y_{1,2}+k_2(h_1-2h_2+h_3)Y_{2,3}+\cdots +k_n(h_{n-1}-2h_n)Y_{n,n+1}.$$

One computes $$h_2-2h_1=(-1)^{h-1}(h+1)!h(n-1)(n-2)\cdots (n-h+1) \neq 0.$$
So if we let $\alpha=\alpha_{1,2}$ and suppose $l$ is large enough for $\mf g$, then the submodule of $\brho(\mf g)$ generated by $v_{2h}$ has a vector with nonzero $\mf g_{-\alpha}$ component, that is, the vector $(adY)^{h+1}v_{2h}$; the vector
$(adY)^{h}v_{2h}$ has nonzero $\mf l_{\alpha}$ component, as $\alpha(h_1H_{1,2}+\cdots +h_nH_{n,n+1})=2h_1-h_2$ which is nonzero.
\end{proof}

\begin{cor}\label{lie Bn}
Assume $l$ is large enough for $\mf g$. For $\mf g$ of type $B_n$, there is a root $\alpha \in \Phi$ such that every irreducible submodule of $\brho(\mf g)$ has a vector with nonzero $\mf l_{\alpha}$ component and a vector with nonzero $\mf g_{-\alpha}$ component.
\end{cor}
\begin{proof}
Let $\mf g=\mf{so}_{2n+1}$. Let $V=k^{2n+1}$ be a vector space equipped with a bilinear form $x_1y_{2n+1}+x_2y_{2n}+\cdots x_{2n+1}y_1$ with matrix $J$. Then $\mf g$ can be identified with
$$\{X \in M_{2n+1}(k)| XJ+JX^t=0\}.$$ The roots of $\mf g$ are
$e_i-e_j$, $e_i+e_j$, $-e_i-e_j$, $\pm e_i$ for $1 \leq i \neq j\leq n$. We choose a set of simple roots $\Delta=\{e_1-e_2,e_2-e_3,\cdots e_{n-1}-e_n, e_n\}$. The $sl_2$-triple associated to $e_i-e_{i+1}$ is
$\{X_i:=X_{i,i+1}-X_{2n-i+1,2n-i+2}, H_i:=H_{i,i+1}+H_{2n-i+1,2n-i+2}, Y_i:=Y_{i,i+1}-Y_{2n-i+1,2n-i+2}\}$, and the $sl_2$-triple associated to $e_n$ is
$\{X_n:=X_{n,n+1}-X_{n+1,n+2}, H_n:=2H_{n,n+1}+2H_{n+1,n+2}, Y_n:=2Y_{n,n+1}-2Y_{n+1,n+2}\}$. 
Let
$$X=\sum_i X_i,$$
$$H=\sum_{1\leq i \leq n-1}i(2n-i+1)H_i+\frac{1}{2}n(n+1)H_n,$$
$$Y=\sum_{1\leq i \leq n-1}i(2n-i+1)Y_i+\frac{1}{2}n(n+1)Y_n.$$
A straightforward calculation shows that $X,H,Y$ form an $sl_2$-triple containing the regular unipotent element $X \in \mf g$.

Corresponding to the exponents $1,3,\cdots ,2n-1$ of $\mf g$, we put
$$v_{2\cdot 1}=X_{1,2}+\cdots +X_{n,n+1}-X_{n+1,n+2}-\cdots -X_{2n,2n+1} \in \mf{so}_{2n+1},$$
$$v_{2 \cdot 3}=X_{1,4}+\cdots +X_{n-1,n+2}-X_{n,n+3}-\cdots -X_{2n-2,2n+1} \in \mf{so}_{2n+1},\cdots,\cdots,$$
$$v_{2 \cdot (2n-1)}=X_{1,2n}-X_{2,2n+1} \in \mf{so}_{2n+1}.$$
Then $\mf g^X=\sum_{i=1,3,\cdots,2n-1}<v_{2i}>$. Let $\alpha=e_1-e_2$, the same calculation as in the proof of Lemma \ref{lie An} gives $(adY)^iv_{2i}$ has a nonzero $\mf l_{\alpha}$ component
and $(adY)^{i+1}v_{2i}$ has a nonzero $\mf g_{-\alpha}$ component for any exponent $i$.

\end{proof}

\subsection{Removing the congruence conditions on $l$} 
In this section, we use a result in \cite{fkp} to remove the congruence condition we have imposed for $G$ of type $B_n, C_n, E_7$ in Theorem \ref{lifting Weyl}. The following theorem is a simplified version of \cite[Theorem 1.3]{fkp}: as we are not considering geometric lifts, we relax the condition at $l$ and only require the right hand side of Wiles' formula (Proposition \ref{Wiles}) to be non-negative. It applies to the residual representations we construct and allows us to deform it to an $l$-adic representation with Zariski-dense image for almost all primes $l$. However, their argument is very different and much more complicated than ours, so we only use it to remove the congruence conditions.

\begin{thm}\label{fkp,main}
Suppose that there is a global deformation condition $\mc L=\{L_v\}_{v \in S}$ consisting of \emph{smooth} local deformation conditions for each place $v \in S$. Let $K=\rats(\brho(\mf g),\mu_l)$.
We assume the following:
\begin{enumerate}
\item $$\sum_{v\in S}(\dim L_v) \geq \sum_{v\in S}h^0(\Gamma_{\rats_v},\brho(\mf g)).$$ 
\item The field $K$ does not contain $\mu_{l^2}$.
\item The groups $H^1(\gal{K}{\rats},\brho(\mf g))$ and $H^1(\gal{K}{\rats},\brho(\mf g)(1))$ vanish.
\item The spaces $\brho(\mf g)$ and $\brho(\mf g)(1)$ are semi-simple $\mb F_l[\galQ]$-modules (equivalently, $k[\galQ]$-modules) having no common $\mb F_l[\galQ]$-sub-quotient, and that neither contains the trivial representation.
\item The space $\brho(\mf g)$ is multiplicity-free as a $\mb F_l[\galQ]$-module.
\end{enumerate}
Then there exists a finite set of primes $Q$ disjoint from $S$, and a lift
$\rho: \Gamma_{\rats, S \cup Q} \to G(\mc O)$ of $\brho$ such that
$\rho$ is of type $L_v$ at all $v \in S$.
\end{thm}

\begin{lem}\label{mu_p2}
Let $k=\mb F_l$ and let $\brho: \galQ \to G(k)$ be as in Section 2.1.4-2.1.7. Then Theorem \ref{fkp,main}, (2)-(5) hold.
\end{lem}
\begin{proof}
By the decomposition of $\brho(\mf g)$ and the proof of Proposition \ref{item2,3}, (3)-(5) hold. It remains to show (2). Since by construction $\brho$ is unramified at $l$, $\rats(\brho(\mf g))$ and $\rats(\mu_l)$ are linearly disjoint over $\rats$. It follows that 
$$\gal{K}{\rats} \cong (Im(\brho)/Z) \times \Zmodu{l}$$ where $Z$ denotes the center of $G(k)$. Assume $K$ contains $\mu_{l^2}$, then
there would be a surjection 
$$\gal{K}{\rats}^{ab} \rsurj \Zmodu{l^2}.$$
On the other hand, we have by the construction of $\brho$ that $Im(\brho)'=Im(\brho)$ for $G$ of type $B_n$ and $E_7$, $Im(\brho)'$ is of index two in $Im(\brho)$ for $G$ of type $C_n$. It follows that the order of $(Im(\brho)/Z)^{ab}$ is at most two, and hence the order of $\gal{K}{\rats}^{ab}$ is at most $2(l-1)$. But this is impossible since $\Zmodu{l^2}$ has order $l(l-1)$ and $l \neq 2$.  

\end{proof}

\section{Simple, simply-connected groups as monodromy groups}

In this section, we prove Theorem \ref{key cases} for $G$ a simple, simply-connected algebraic group. Recall that there are two different constructions for the residual representation $\brho: \galQ \to G(k)$: one has image a large index subgroup of $N_G(T)_k$ with the properties that $\brho$ is unramified at $l$ and $ad\brho(c)$ is nontrivial; the other factors through a principal $GL_2$ such that $\brho(c)=\rho^{\vee}(-1)$.

\subsection{Local deformation conditions}

We need to define several local deformation conditions for deforming the mod $p$ representations.

\subsubsection{The archimedean place}
Recall that in Section 2.1.3 through Section 2.1.5, we construct
the residual representations by first realizing $S_n$ or $A_n$ as a Galois group over $\rats$ and then repeatedly applying Theorem \ref{2.5} to build the Galois extension realizing $N$ or a subgroup of it over $\rats$. We write $c$ for the nontrivial element in $\Gamma_{\reals}$, the complex conjugation.

\begin{prop}\label{4.1}
Let $G$ be of classical type and $\brho: \galQ \to G(k)$ be as in Section 2.1.2-2.1.5. In particular, $ad\brho(c)$ is nontrivial. 
Then
\begin{enumerate}
\item For $G$ of type $A_{n-1}$, $h^0(\Gamma_{\reals}, \brho(\mf g)) \leq n^2-2n+1$.
\item for $G$ of type $B_{n}$, $h^0(\Gamma_{\reals}, \brho(\mf g)) \leq 2n^2-3n+2$.
\item for $G$ of type $C_{n}$, $h^0(\Gamma_{\reals}, \brho(\mf g)) \leq 2n^2-3n+4$.
\item for $G$ of type $D_{n}$, $h^0(\Gamma_{\reals}, \brho(\mf g)) \leq 2n^2-5n+4$.
\end{enumerate}
\end{prop}
\begin{proof}

Let $f$ be the number of root vectors that are fixed by $\brho(c)$. Let us recall that $ad\brho(c)$ is nontrivial by Theorem \ref{2.4}.

If $G$ is of type $A_{n-1}$, we have $G(k)\cong \op{SL}(V)$ with $V=k^n$. Let $d:=\dim V^{\brho(c)}$. Then $$f=2\big(\binom{d}{2}+\binom{n-d}{2}\big).$$

If $G$ is of type $B_n$, $G(k)/\mu_2\cong\op{SO}(V)$ with $V=k^{2n+1}$ a $k$-vector space equipped with the non-degenerate symmetric bilinear form $x_1y_{2n+1}+x_2y_{2n}+\cdots +x_{2n+1}y_1$. We may assume $\brho(c)$ is conjugate to $diag(\epsilon_1,\cdots,\epsilon_n,1,\epsilon_n,\cdots,\epsilon_1)$ in $\op{SO}(V)$  
and let $d$ be the number of 1's among $\epsilon_1,\cdots,\epsilon_n$. Then $$f=4\big(\binom{d}{2}+\binom{n-d}{2}\big)+2d.$$

If $G$ is of type $C_n$, $G(k)\cong\op{Sp}(V)$ with $V=k^{2n}$ a $k$-vector space equipped with the non-degenerate alternating bilinear form $x_1y_{2n}+\cdots x_ny_{n+1}-x_{n+1}y_n-\cdots -x_{2n}y_1$. We may assume $\brho(c)$ is conjugate to $diag(\epsilon_1,\cdots,\epsilon_n,\epsilon_n,\cdots,\epsilon_1)$  
and let $d$ be the number of 1's among $\epsilon_1,\cdots,\epsilon_n$. Then $$f=4\big(\binom{d}{2}+\binom{n-d}{2}\big)+2n.$$

If $G$ is of type $D_n$, a quotient of $G(k)$ is isomorphic to $\op{SO}(V)$ with $V=k^{2n}$ a $k$-vector space equipped with the non-degenerate symmetric bilinear form $x_1y_{2n}+x_2y_{2n-1}+\cdots +x_{2n}y_1$. We may assume $\brho(c)$ is conjugate to $diag(\epsilon_1,\cdots,\epsilon_n,\epsilon_n,\cdots,\epsilon_1)$ in $\op{SO}(V)$  
and let $d$ be the number of 1's among $\epsilon_1,\cdots,\epsilon_n$. Then 
$$f=4\big(\binom{d}{2}+\binom{n-d}{2}\big).$$

By the construction of $\brho$ (see the paragraph before Remark \ref{2.8} in Section 2.1.3 for type $A_n$; for other types, see Sections 2.1.4, 2.1.5 and 2.1.6), $ad\brho(c)$ is nontrivial, which implies $0<d<n$. So $f$ attains its maximum when $d=n-1$. Since 
$$h^0(\Gamma_{\reals}, \brho(\mf g))=rk(\mf g)+f,$$
the upper bounds can then be computed easily.
\end{proof}

\begin{prop}\label{4.2}
Let $G$ be of type $E_7$ and $\brho: \galQ \to G(k)$ be as in Section 2.1.6. In particular, $ad\brho(c)$ is nontrivial. Then 
$h^0(\Gamma_{\reals}, \brho(\mf g)) \leq 7+126-14=119$.
\end{prop}

\begin{proof}
Suppose that $\brho(c) \in T(k)$ for a maximal torus $T$ split over $k$. Let $\Phi=\Phi(G,T)$, and $\mf g_{\Phi}=\sum_{\alpha\in \Phi} \mf g_{\alpha}$ be the $k$-subspace of $\brho(\mf g)$ generated by all root vectors. The Lie algebra $\mf t$ of $T(k)$, which has $k$-dimension 7, is clearly fixed by $ad\brho(c)$. Thus, it suffices to show that the $-1$-eigenspace of $ad\brho(c)|\mf g_{\Phi}$ has $k$-dimension at least 14. 
We consider $ad\brho(c)|\mf g_{\Phi'}$ where $\Phi' \subset \Phi$ is of type $A_7$. This action is nontrivial. By the $A_n$-calculation in the proof of Proposition \ref{4.1} (letting $n=7$), the $-1$-eigenspace of $ad\brho(c)|\mf g_{\Phi'}$ has dimension at least twice the rank of $\Phi'$, proves the proposition. 
\end{proof}

The following lemma is clear.
\begin{lem}\label{4.3}
The dimension $$\dim_k \big( \op{Sym}^{2n}(k^2)\otimes \op{det}^{-n} \big)^{diag(1,-1)}$$ equals 
$n$ when $n$ is odd, and $n+1$ when $n$ is even.
\end{lem}

\begin{cor}\label{4.4}
Let $\brho: \galQ \to G(k)$ be as in Section 2.2. Then 
$h^0(\Gamma_{\reals}, \brho(\mf g))=4$ for $G=\op{SL}_3$,
$h^0(\Gamma_{\reals}, \brho(\mf g))=9$ for $G=\op{Spin}_7$,
$h^0(\Gamma_{\reals}, \brho(\mf g))=38$ for $G=E_6^{sc}$.
\end{cor}
\begin{proof}
This follows from Lemma \ref{4.3} and Proposition \ref{kostant}. Note that the complex conjugation maps to $diag(1,-1)$ in $\op{GL}_2$ because the representation $r_{f,\lambda}$ in the last paragraph of Section 2 is odd.
\end{proof}

\subsubsection{The place $l$}
As we are not looking for geometric $l$-adic Galois representations in this paper, we impose \emph{no condition} at the place $l$. So the tangent space is $H^1(\Gamma_l, \brho(\mf g))$. By the local Euler characteristic formula, 
$$h^1(\Gamma_l, \brho(\mf g))=h^0(\Gamma_l, \brho(\mf g))+h^2(\Gamma_l, \brho(\mf g))+\dim_k \mf g.$$ 

\begin{lem}\label{4.5} Let $\brho$ be as in Section 2.1 or Section 2.2. Then
$h^2(\Gamma_l, \brho(\mf g))=0$ for large enough primes $l$.
\end{lem}
\begin{proof}
By local duality, it suffices to show that
$h^0(\Gamma_l, \brho(\mf g)(1))=0$.
For the representation $\brho$ in Section 2.1, $\brho(I_{\rats_l})$ is trivial by construction but $\bkp(I_{\rats_l})$ is nontrivial, so $\brho(\mf g)(1)^{I_{\rats_v}}$ is trivial. In particular, $h^0(\Gamma_l, \brho(\mf g)(1))=0$. 
For the representation $\brho$ in Section 2.2, Proposition \ref{kostant} and the lemma below imply $h^0(\Gamma_l, \brho(\mf g)(1))=0$. 
\end{proof}

\begin{lem}\label{4.6}
We have $h^0(\Gamma_l,\op{Sym}^{2m}(\bar r_f)\otimes \op{det}(\bar r_f)^{-m} \otimes \bkp)=0$ for $m\geq 1$ and large enough primes $l$ (relative to $m$).
\end{lem}
\begin{proof}
The argument is similar to the proof of \cite{wes04}, Proposition 4.4. Let $K$ be a finite extension of $\Ql$ with ring of integers $\mc O$ and residue field $k$. Let $v: K^{\times} \to \rats$ be the valuation on $K$, normalized so that $v(l)=1$. We briefly recall the setting in Section 4.1 of \cite{wes04}. For $a<b$, let $\mc{MF}^{a,b}(\mc O)$ denote the category of filtered Dieudonn\'e $\mc O$-modules $D$ equipped with a decreasing filtration of $\mc O$-modules $\{D_i\}_{i \in \ints}$ and a family of $\mc O$-linear maps $\{f_i: D^i \to D\}$ satisfying $D^a=D$ and $D^b=0$ (see \cite[Definition 4.1]{wes04}). Let 
$\mc G^{a,b}(\mc O)$ denote the category of finite type $\mc O$-module subquotients of crystalline $K$-representations $V$ with $D_{crys}^a(V)=D_{crys}(V)$ and $D_{crys}^b(V)=0$.  Fontaine-Laffaille define a functor 
$$\mc U: \mc{MF}^{a,a+l}(\mc O) \to \mc G^{a,a+l}(\mc O)$$
satisfying a list of properties including $\mc U$ that is stable under formation of sub-objects and quotients, and compatible with tensor products for $l$ large enough. In particular, $\mc U$ is compatible with symmetric powers for large enough $l$.
Let $\varepsilon: \Gamma_l \to \mc O^{\times}$ be an unramified character of finite order and let $\mc O(\varepsilon)$ denote a free $\mc O$-module of rank 1 with $\Gamma_l$-action via $\varepsilon$. Then $\mc O(\varepsilon) \in \mc G^{0,1}(\mc O)$, so that there is $D_{\varepsilon} \in \mc{MF}^{0,1}(\mc O)$ such that $\mc U(D_{\varepsilon})=\mc O(\varepsilon)$. This $D_{\varepsilon}$ is a free $\mc O$-module of rank one with $D_{\varepsilon}=D_{\varepsilon}^0$ and $f_0: D_{\varepsilon}^0 \to D_{\varepsilon}$ the multiplication by $\varepsilon^{-1}(l):=\varepsilon^{-1}(\frob{l})$ where the Frobenius element is arithmetic.

Let $f$ be a newform of weight 3, level $N$, and character $\varepsilon$. Let $K\supset E_{\lambda}$ and let $r_f:=r_{f,\lambda}: \galQ \to \gl{2}{K}$ be the Galois representation associated to $f$ of weight $k$ (see Example \ref{sl2 mod form}). We fix an embedding $\Gamma_l \to \galQ$, and let $V_f$ be a two dimensional $K$-vector space on which $\Gamma_l$ acts via $r_f|\Gamma_l$, and fix a $\Gamma_l$-stable $\mc O$-lattice $T_f \subset V_f$. If $l$ does not divide $N$, then $V_f$ is crystalline and $T_f \in \mc G^{0,3}(\mc O)$. Thus for $l>k$ there exists $D_f \in \mc{MF}^{0,3}(\mc O)$ with $\mc U(D_f) \cong T_f$. The filtration on $D_f$ satisfies
$\op{rk}_{\mc O}(D_f^i)=2$ if $i\leq 0$, $\op{rk}_{\mc O}(D_f^i)=1$ if $1 \leq i\leq 2$, and $\op{rk}_{\mc O}(D_f^i)=0$ if $i\geq 3$. Choose an $\mc O$-basis $x,y$ of $D_f$ with $x$ an $\mc O$-generator of $D_f^1$. Let $a,b,c,d \in \mc O$ be such that
$f_0x=ax+by$, $f_0y=cx+dy$. Then $a+d=a_l$ and $ad-bc=l^2\varepsilon(l)$. We have $v(a), v(b) \geq 2$.

Let $\bar r_f: \Gamma_l \to \gl{2}{k}$ be the Galois representation $T_f/\lambda T_f$. Since $\det r_f=\kappa^2\varepsilon$, we have
$$\op{Sym}^{2m}(T_f)\otimes \op{det}(T_f)^{-m} \otimes \kappa=(\op{Sym}^{2m}(T_f)\otimes \mc O(\varepsilon^{-m}))(1-2m).$$
When $l$ is large enough relative to $m$, by \cite{fm87}, Proposition 1.7, we can take 
$$D=(\op{Sym}^{2m}(D_f)\otimes D_{\varepsilon^{-m}})(2m-1).$$
Further, since $(\op{Sym}^{2m}(T_f)\otimes \op{det}(T_f)^{-m} \otimes \kappa)/\lambda$ is a realization of $\op{Sym}^{2m}(\bar r_f)\otimes \op{det}(\bar r_f)^{-m} \otimes \bkp$, we have
$$H^0(\Gamma_l,\op{Sym}^{2m}(\bar r_f)\otimes \op{det}(\bar r_f)^{-m} \otimes \bkp)=\op{ker}(1-f_0: D^0/\lambda D^0 \to D/\lambda D).$$

By the definition of Tate twists and tensor products of filtered Dieudonn\'e $\mc O$-modules, we have
$$D^0=(\op{Sym}^{2m}(D_f)\otimes D_{\varepsilon^{-m}})^{2m-1}$$
$$=\sum_{i_1+\cdots+i_{2m}+j=2m-1}D_f^{i_1}\cdots D_f^{i_{2m}}\cdot D_{\varepsilon^{-m}}^j$$
$$=\sum_{i_1+\cdots+i_{2m}=2m-1}D_f^{i_1}\cdots D_f^{i_{2m}}\cdot D_{\varepsilon^{-m}}^0.$$
To make the sum nonzero, there must be at least $m$ indices that are greater than or equal to 1, so at least $m$ indices must be two since $x \in D_f^2$ as well as $D_f^1$. It follows that $\{x^iy^{2m-i-1}w|i\geq m\}$ is an $\mc O$-basis of $D^0$, where $w$ is an $\mc O$-generator of $D_{\varepsilon^{-m}}$.
We compute
$$f_0(x^iy^{2m-i-1}w)=\frac{\varepsilon^m(l)}{l^{2m-1}}(ax+by)^i(cx+dy)^{2m-i-1}.$$
Since $v(a),v(b) \geq 2$ and $i \geq m$, all the coefficients of $x$ and $y$ have positive valuations. Therefore, $f_0(x^iy^{2m-i-1}w) \equiv 0$ modulo $\lambda$, which implies $f_0: D^0/\lambda D^0 \to D/\lambda D$ is zero. It follows that  
$H^0(\Gamma_l,\op{Sym}^{2m}(\bar r_f)\otimes \op{det}(\bar r_f)^{-m} \otimes \bkp)$ is trivial.
\end{proof}

\begin{cor}\label{4.7}
$h^1(\Gamma_l, \brho(\mf g))=h^0(\Gamma_l, \brho(\mf g))+\dim_k \mf g$.
\end{cor}

\subsubsection{A zero-dimensional deformation}
In order to maximize the Zariski-closure of the image of the $l$-adic lift of the residual representation, we need to impose a simple local deformation condition at some unramified place.

Suppose that $p\neq l$, $F$ is a finite extension of $\Qp$, and $\brho: \Gamma_F \to G(k)$ is an unramified representation. Let $g \in G(\mc O)$ be a lift of $\brho(\frob{p})$.

\begin{defn}\label{trivial con}
Define $$\op{Lift}_{\brho}^{g}: \op{CNL}_{\mc O} \to \mathbf{Sets}$$
such that for a complete local noetherian $\mc O$-algebra $R$, $\op{Lift}_{\brho}^{g}(R)$ consists of all lifts 
$$\rho: \Gamma_F \to G(R)$$ of $\brho$ such that $\rho$ is unramified and $\rho(\frob{p})$ is $\hat G(R)$-conjugate to $g$. 
\end{defn}

So the tangent space is zero-dimensional and when $\op{Lift}_{\brho}^{g}$ is a local deformation condition, it is clearly smooth. But for a given $g$, $\op{Lift}_{\brho}^{g}$ may not be representable. But at least we have
\begin{prop}\label{representability}
Suppose that $G$ is simply-connected. Let $\bar g$ (resp. $g$) be a regular semisimple element of $G(\overline{\mb F}_l)$ (resp. $G(\bQl)$). Then 
$\op{Lift}_{\brho}^{g}$ is representable.
\end{prop}
\begin{proof}
By Schlessinger's criterion, it suffices to show the following:
for any $A\rsurj B$ in $\op{CLN}_{\mc O}$ with kernel $I$ for which $I\cdot \mf m_{A}=0$, the induced map
$$Z_G(g)(A) \to Z_G(g)(B)$$
is surjective. The group $Z_G(g)$ is a scheme over $\mc O$, we denote the structure map by $$f: Z_G(g) \to \op{Spec}\mc O.$$ We need to show that $Z_G(g)$ is a smooth $\mc O$-scheme. It suffices to show that 
\begin{itemize}
\item The map $f$ is flat over $\mc O$.
\item The generic fiber and the special fiber of $f$ are smooth of the same dimension.
\end{itemize}
Because $G$ is simply-connected and $\bar g$ (resp. $g$) is regular semisimple, $Z_G(g)(\mc O/\lambda)$ (resp. $Z_G(g)(\op{Frac}\mc O)$) is a connected maximal torus of $G(\mc O/\lambda)$ (resp. $G(\op{Frac}\mc O)$) with dimension the rank of $G$. The second bullet follows.

To show the first bullet, note that $f$ has a section, that is, $Z_G(g)$ has an $\mc O$-point (for example, the element $g \in G(\mc O)$ itself). Moreover, by the previous paragraph, the generic fiber and the special fiber of $f$ are both irreducible, reduced and have the same dimension. It follows from Proposition 6.1 of \cite{gy} that $f$ is flat.

\end{proof}

\subsubsection{Steinberg deformations}
In \cite[Section 4.3]{pat16}, a local deformation condition of `Steinberg type' is taken at a place in order to obtain a regular unipotent element in the image of the $l$-adic lift. \emph{We will only need this in deforming those $\brho$ constructed from the principal $\op{GL}_2$}. We refer the reader to \cite[Section 4.3]{pat16} for the definition and properties of the Steinberg deformation condition. The dimension of the tangent space equals $h^0(\Gamma_v, \brho(\mf g))$.  

\subsubsection{Minimal prime to $l$ deformations}
This deformation condition is well-known; see \cite[Section 4.4]{pat16} for its definition. We will use this deformation condition at places $v\neq l$ for which $\brho(I_{\rats_v})$ is nontrivial and $\brho(\Gamma_{\rats_v})$ has order prime to $l$.
The tangent space is $H^1(\Gamma_v/I_v, \brho(\mf g)^{I_v})$, whose dimension is $h^0(\Gamma_v, \brho(\mf g))$.

\subsection{Deforming mod $p$ Galois representations}
In this section, we specify the global deformation condition and compute the Wiles formula, then use the results in Section 3 to prove Theorem \ref{key cases}.
Let us recall Wiles' formula, for a proof, see \cite{pat16}, Proposition 9.2.
\begin{prop}\label{Wiles}
Let $M$ be a finite-dimensional $k$-vector space with a continuous $\galQ$ action unramified outside a finite set of places $S$. Let
$\mc L=\{L_v\}_{v \in S}$ (resp. $\mc L^{\perp}=\{L_v^{\perp}\}_{v \in S}$) be a Selmer system (resp. dual Selmer system) for $M$. Then
$$\sel{\mc L}{\Gamma_S}{M}-\sel{\mc L^{\perp}}{\Gamma_S}{M^{\vee}}=h^0(\Gamma_S, M)-h^0(\Gamma_S, M^{\vee})+\sum_{v \in S} (\dim_k L_v-h^0(\Gamma_v, M)).$$
\end{prop}

We will compute the right hand side of the identity for $M=\brho(\mf g)$ or $\brho(\mf t)$ and for
a global deformation condition to be specified below. For $\brho(\mf g)$ from either Section 2.1 or Section 2.2, note that $h^0(\Gamma_S, \brho(\mf g))=h^0(\Gamma_S, \brho(\mf g)(1))=0$. 

\subsubsection{Weyl group case}
For $G$ a simple, simply connected group of classical type or type $E_7$, let $\brho: \galQ \to G(k)$ be as in Section 2.1. Here we exclude the $A_1, A_2, B_3$ cases.
We impose no condition at $v=l$ which is liftable by Lemma \ref{4.5}, and impose the minimal prime to $l$ condition at $v \in S-\{\infty,l\}$ (note that $\brho(\galQ)$ has order prime to $l$ by our construction). Moreover, we will find a prime $p \notin S$ for which $\brho(\frob{p})$ is regular semisimple, together with a regular semisimple lift $g \in G(\mc O)$ of $\brho(\frob{p})$. Then we take the deformation condition $\op{Lift}_{\brho|_{\Gamma_p}}^g$ at $p$.

\begin{lem}\label{Wiles g}
$\sum_{v \in S} \dim_k L_v \geq \sum_{v\in S} h^0(\Gamma_v, \brho(\mf g))$.
\end{lem}
\begin{proof}
This follows directly the local computations in Section 4.1, for example Proposition \ref{4.1}, Corollary \ref{4.7}, etc. We record here a lower bound for $\sum_{v \in S} \dim_k L_v-\sum_{v\in S} h^0(\Gamma_v, \brho(\mf g))$. For $G=\op{SL}_n$, it is $n-1$; for $G=\op{Spin}_{2n+1}$, it is $3n-2$; for $G=\op{Sp}_{2n}$, it is $3n-4$; for $G=\op{Spin}_{2n}$, it is $3n-4$; and for $G=E_7$, it is $7$.  
\end{proof}

\begin{lem}\label{Wiles t}
$\sum_{v \in S} \dim_k (L_v \cap H^1(\Gamma_v, \brho(\mf t))) \leq \sum_{v\in S} h^0(\Gamma_v, \brho(\mf t))$.
\end{lem}
\begin{proof}
For $v \notin \{\infty,l,p\}$, $L_v$ corresponds to the minimal prime to $l$ deformation condition and we have $$\dim_k (L_v \cap H^1(\Gamma_v, \brho(\mf t)))=\dim_k H^1(\Gamma_v/I_v, \brho(\mf t)^{I_v})=h^0(\Gamma_v, \brho(\mf t)).$$ So it suffices to compare both sides for $v\in \{\infty, l,p\}$. The left hand side subtracting the right hand side equals
$$(0-h^0(\Gamma_{\reals},\brho(\mf t)))+(h^1(\Gamma_l,\brho(\mf t))-h^0(\Gamma_l,\brho(\mf t)))+(0-h^0(\Gamma_p,\brho(\mf t))).$$
By local duality and Lemma \ref{4.5}, 
$$h^1(\Gamma_l,\brho(\mf t))-h^0(\Gamma_l,\brho(\mf t))=\dim_k\mf t.$$
Combining this with the identity 
$$h^0(\Gamma_p,\brho(\mf t))=\dim_k\mf t,$$
we see that the difference is $-h^0(\Gamma_{\reals},\brho(\mf t)) \leq 0$.
\end{proof}

Let us make the following observation which is from \cite{pat16}, Lemma 7.7. It will be used frequently in the proof of Proposition \ref{weyl mono}, \ref{pr gl2 mono}, \ref{sl2 A5}: suppose that $\brho: \galQ \to G(k)$ is a continuous representation with a continuous lift $\rho: \galQ \to G(\mc O)$. Let $G_{\rho}$ be the Zariski closure of $G(\mc O)$ in $G(\bQl)$. Then $\op{Lie}(G_{\rho})$, $\op{Lie}(G_{\rho}) \cap \mf g_{\mc O}$, and $(\op{Lie}(G_{\rho}) \cap \mf g_{\mc O})\otimes_{\mc O} k$ are $\galQ$-modules. Moreover, the last one is a submodule of $\brho(\mf g)$ and thus is a direct sum of some irreducible summands of $\brho(\mf g)$. If $\op{Lie}(G_{\rho})=\mf g(\bQl)$ (which is equivalent to $(\op{Lie}(G_{\rho}) \cap \mf g_{\mc O})\otimes_{\mc O} k=\mf g$), then $G_{\rho}=G(\bQl)$ (since $G$ is connected).

\begin{lem}\label{dense}
Let $G$ be a semisimple algebraic group defined over $\mc O$ and let $C$ be a proper subvariety of $G$. Let $g \in G(\mc O)$ and $H=g\hat G(\mc O)$ be the corresponding $\hat G(\mc O)$-coset of $G(\mc O)$. Then there is a regular semisimple element in $H-C(\bQl)$.  
\end{lem}
\begin{proof}
Let $V$ be the union of $C(\bQl)$ and the set of elements of $G(\bQl)$ that are not regular semisimple. Then $V$ is a proper Zariski-closed subset of $G(\bQl)$ as the set of regular semisimple elements is Zariski-open (see for example \cite[Theorem 2.5]{hum95}). On the other hand, $H$ is Zariski-dense in $G(\bQl)$. If there were no regular semisimple element in $H-C(\bQl)$, then $H \subset V$ and $H$ could not be Zariski-dense, a contradiction.
\end{proof}

\begin{prop}\label{weyl mono} Let $G$ be a simple, simply-connected group of classical type (excluding type $A_1, A_2$ and $B_3$) or type $E_7$. Then for almost all primes $l$ when $\Phi$ is of type $A_n$ or $D_n$, for almost all primes $l\equiv 1(4)$ when $\Phi$ is of type $B_n$ or $C_n$, and for almost all primes $l\equiv 1(3)$ when $\Phi$ is of type $E_7$,   
there are $l$-adic lifts
$$\rho: \galQ \to G(\mc O)$$ of $\brho: \galQ \to G(k)$ defined in Section 2.1
with Zariski-dense image in $G(\bQl)$. 
\end{prop}
\begin{proof}
By Lemma \ref{Wiles g} and \ref{Wiles t}, we can apply Theorem \ref{lifting Weyl} to obtain a lift $\rho: \galQ \to G(\mc O)$ satisfying the prescribed local conditions.
The condition at $p$ implies that $G_{\rho}$ has infinitely many elements and so $\op{Lie}(G_{\rho})$ is nontrivial.

If $G$ is of type $A_n$ or $D_n$, by 
Proposition \ref{2.9} and \ref{2.13}, $(Lie(G_{\rho})\cap\mf g_{\mc O})\otimes_{\mc O}k$ (as a Lie subalgebra of $\mf g$) is then either $\mf t$ or $\mf g$. By the previous lemma, there exist a regular semisimple element $g \in G(\mc O)$ such that $\bar g=\brho(\frob{p})$ and $g \notin N_G(T)(\bQl)$. Imposing $\op{Lift}_{\brho|p}^g$ at $p$, we obtain $G_{\rho} \nsubseteq N_G(T)(\bQl)$, which implies $(Lie(G_{\rho})\cap\mf g_{\mc O})\otimes_{\mc O}k$ cannot be $\mf t$.

If $G$ is of type $B_n$, by 
Proposition \ref{2.13}, $(Lie(G_{\rho})\cap\mf g_{\mc O})\otimes_{\mc O}k$ (as a Lie subalgebra of $\mf g$) is 
$\mf t, \mf t \oplus \mf g_l \cong \mf{so}_{2n}$ or $\mf g$. 
Let $H$ be an algebraic subgroup of $G$ containing $T$ such that
$(Lie(H)\cap\mf g_{\mc O})\otimes_{\mc O}k=\mf t \oplus \mf g_l$ (there are finitely many of them). Let $C$ be the union of $N_G(T)$ and all such $H$, which is a proper subvariety of $G$. 
By the previous lemma, there exist a regular semisimple element $g \in G(\mc O)$ such that $\bar g=\brho(\frob{p})$ and $g \notin C(\bQl)$. Imposing $\op{Lift}_{\brho|p}^g$ at $p$, we obtain $G_{\rho} \nsubseteq C(\bQl)$, which implies $(Lie(G_{\rho})\cap\mf g_{\mc O})\otimes_{\mc O}k$ cannot be $\mf t$ or $\mf t \oplus \mf g_l$.

If $G$ is of type $C_n$, by 
Proposition \ref{2.11}, $(Lie(G_{\rho})\cap\mf g_{\mc O})\otimes_{\mc O}k$ (as a Lie subalgebra of $\mf g$) is 
$\mf t, \mf t \oplus \mf g_l \cong (\mf{sl}_{2})^n$ or $\mf g$. 
The same argument as for type $B_n$ enables us to impose a suitable condition $\op{Lift}_{\brho|p}^g$ at $p$ in order to force $(Lie(G_{\rho})\cap\mf g_{\mc O})\otimes_{\mc O}k=\mf g$.

Finally, suppose that $G=E_7^{sc}$. By Proposition \ref{e7,3}, $(Lie(G_{\rho})\cap\mf g_{\mc O})\otimes_{\mc O}k$ (as a Lie subalgebra of $\mf g$) is either $\mf t$, $\mf t \oplus \mf g_a \cong \mf{sl}_7$ or $\mf g$. We can force $(Lie(G_{\rho})\cap\mf g_{\mc O})\otimes_{\mc O}k$ to be $\mf g$ in the same way as above. 
\end{proof}

\begin{rmk}\label{non-conj 1}
As we have flexibilities in choosing $g \in G(\mc O)$ lifting $\brho(\frob{p})$, it is easy to see that there are infinitely many lifts $\rho$ that are non-conjugate in $G^{ad}$.
\end{rmk}

\subsubsection{Principal $\op{GL}_2$ case}
For $G$ a simply connected group of one of the following types: $A_2, B_3, E_6$, let $\brho: \galQ \to G(k)$ be as in Section 2.2.

We begin with the following proposition due to Tom Weston (\cite{wes04}, Proposition 5.3):

\begin{prop}\label{wes steinberg}
Let $\pi=\pi_f$ be a cuspidal automorphic representation
corresponding to a holomorphic eigenform $f$ of weight at least 2. Assume that for some prime
$p$, $\pi_{p}$ is isomorphic to a twist of the Steinberg representation of $\gl{2}{\Qp}$. Then for almost all
$\lambda$, the local Galois representation $\bar r_{f,\lambda}|_{\Gamma_p}$ (in the notation of Example \ref{sl2 mod form})
has the form 
$$\bar r_{f,\lambda}|_{\Gamma_p} \sim \bp \chi\bkp & *\\ 0 & \chi \ep$$
where the extension * in $H^1(\Gamma_p, k_{\lambda}(\bkp))$ is non-zero.
\end{prop}

Let $f$ be a non-CM weight 3 cuspidal eigenform that is a newform of level $\Gamma_1(p) \cap \Gamma_0(q)$
for some primes $p$ and $q$; the nebentypus of $f$ is a character $\varepsilon: \Zmodu{pq} \to \Zmodu{p} \to \cmplx^{\times}$. 
Such a form $f$ exists, for example, \cite{lmf}, 15.3.7.a and 15.3.11.a. Note that the automorphic representation $\pi_f$ associated to $f$ is Steinberg at $p$. Proposition \ref{wes steinberg} together with the definition of principal $\op{GL}_2$ then imply $\brho|_{\Gamma_p}$ is Steinberg in the sense of \cite[Definition 4.13]{pat16}. 
At $p$ we take the Steinberg deformation condition. As $\pi_f$ is a principal series at $q$, $\brho(I_q)$ has order prime to $l$. We then use the minimal prime to $l$ deformation condition at $q$. At $l$ we impose no condition. Moreover, choose an element $\sigma \in \galQ$ such that $\brho(\sigma)$ is regular semi-simple in $T(k)$ together with a lift $g \in T(\mc O)$ such that $\alpha(g)$, $\alpha\in \Delta$ are distinct. 
By Chebotarev's density theorem, there is a prime $r \notin \{\infty,l,p,q\}$ such that $\frob{r}=\sigma$. We then take $\op{Lift}_{\brho|_{\Gamma_r}}^g$ at $r$. Let $\mc L$ be the Selmer system associated to the above local deformations.

\begin{lem}\label{Wiles A2B3E6} The right hand side of Wiles' formula is $2$ for $A_2$, $9$ for $B_3$, and $34$ for $E_6$. 
\end{lem}
\begin{proof}
This follows from Corollaries \ref{4.4}, \ref{4.7} and Proposition \ref{Wiles}.
\end{proof}

\begin{prop}\label{pr gl2 mono} For $G=\op{SL}_{3}, \op{Spin}_{7}$ or $E_6^{sc}$ and for almost all primes $l$,  
there are $l$-adic lifts
$$\rho: \galQ \to G(\mc O)$$ of $\brho: \galQ \to G(k)$ defined as in Section 2.2 with Zariski-dense image in $G(\bQl)$.
\end{prop}
\begin{proof}
The proof is very similar to the proof of \cite{pat16}, Theorem 7.4, so we skip a few details here.

We first show that Theorem \ref{pat ram} applies to $\brho$. Item 1 is satisfied by Lemma \ref{Wiles A2B3E6}; Item 2 and 3 are satisfied by the proof of \cite{pat16}, Theorem 7.4; for Item 4, take $\sigma \in \galQ$ such that $\brho(\sigma)=2\rho^{\vee}(a)$ is regular with $1\neq a \in \Zmodu{l}$ and $\bkp(\sigma)=a^2$ (which is possible, again, see the proof of \cite{pat16}, Theorem 7.4). It follows that Item (a) is satisfied for any simple root $\alpha$. Item (b) and (c) are also satisfied by Lemma \ref{lie excep}, \ref{lie An} and \ref{lie Bn}. 

Therefore, we can deform $\brho$ to a continuous representation
$\rho: \galQ \to G(\mc O)$ satisfying the prescribed local conditions on $S$ and the Ramakrishna condition on a set of auxiliary primes disjoint from $S$. We write $G_{\rho}$ for the Zariski closure of the image of $\rho$ in $G(\bQl)$. By \cite{pat16}, Lemma 7.7, $G_{\rho}$ is reductive. By Proposition \ref{wes steinberg}, the Steinberg condition at $p$ ensures that $G_{\rho}$ contains a regular unipotent element (see the proof of \cite{pat16}, Theorem 8.4). By a theorem of Dynkin (\cite[Theorem A]{ss97}), $G_{\rho}$ is then of type $A_1$ or $A_2$ for $G=\op{SL}_{3}$, type $A_1, G_2$ or $B_3$ for $G=\op{Spin}_{7}$, and type $A_1, F_4$, or $E_6$ for $G=E_6^{sc}$. But $\alpha(\rho(\frob{r}))$, $\alpha \in \Delta$ are distinct, so $G_{\rho}=G(\bQl)$ in all three cases (see the proof of \cite{pat16}, Lemma 7.8).   
\end{proof}

\begin{rmk}\label{non-conj 2}
As we have flexibilities in choosing $g \in G(\mc O)$ lifting $\brho(\frob{r})$, it is easy to see that there are infinitely many lifts $\rho$ that are non-conjugate in $G^{ad}$.
\end{rmk}

\subsubsection{$\op{SL}_2$}
The alternating group $A_n$ admits a unique nontrivial central extension $\tilde A_n$ by $\Zmod{2}$ for $n \neq 6,7$. By a result of N. Vila and J.-F. Mestre (which was proven independently, see \cite{ser:tgt}), $\tilde A_n$ can be realized as a Galois group over $\rats$. In particular, we get a surjection $\bar r: \galQ \rsurj \tilde A_5$. On the other hand, $\tilde A_5$ can be described as follows: the symmetries of an icosahedron induce a 3-dimensional irreducible faithful representation of $A_5$, i.e., there is an injective homomorphism $A_5 \to \op{SO}(3)$. The pullback of $A_5$ along the two-fold covering map $\op{SU}(2) \rsurj \op{SO}(3)$ is a nontrivial central extension of $A_5$ by $\Zmod{2}$, hence is isomorphic to $\tilde A_5$. In particular, we get an embedding $\tilde A_5 \to \sl{2}{\cmplx}$. As the matrix entries of the image lie in a finite extension of $\rats$, we can choose a finite extension $k$ of $\mb F_l$ for which there is an embedding 
$\tilde A_5 \to \sl{2}{k}$. Precomposing it with $\bar r$, we obtain a representation $\galQ \to \sl{2}{k}$ which we denote by $\brho$. It is easy to see that the adjoint module $\brho(\mf{sl}_2(k))$ is irreducible.

Let $S$ be a finite set of places containing the archimedean place and all places where $\brho$ is ramified. We impose no condition at $l$, and take the minimal prime to $l$ deformation condition at all other places in $S$, which is legitimate since the residual image has order prime to $l$ for $l>120$. Let $\Sigma \in \tilde A_5$ be an element of order 4, whose image in $\sl{2}{k}$ is conjugate to $diag(\sqrt{-1},-\sqrt{-1})$. As $\tilde A_5$ has trivial abelian quotient, $\rats(\mu_l)$ and $\rats(\brho)$ are linearly disjoint over $\rats$. So there is an element $\sigma \in \galQ$ such that
$\brho(\sigma)=\Sigma$ and $\bkp(\sigma)=-1$. By Chebotarev's density theorem, there is a prime $p \notin S$ for which $\frob{p}=\sigma$. Therefore, $\brho|_{\Gamma_p}$ is of Steinberg type and we take the Steinberg deformation condition at $p$.
 
\begin{prop}\label{sl2 A5} For $\brho: \galQ \to \sl{2}{k}$ defined as above and for almost all primes $l$, 
there is an $l$-adic lift
$$\rho: \galQ \to \sl{2}{\mc O}$$ of $\brho$ with Zariski-dense image in $\sl{2}{\bQl}$.
\end{prop}
\begin{proof}
We first show that Theorem \ref{pat ram} applies to $\brho$. Item 1 is satisfied: the left hand side of the inequality equals the right hand side; Item 2 is satisfied since $|\gal{K}{\rats}|$ has order prime to $l$ by the definition of $\brho$; Item 3 is satisfied since $\brho(\mf g)$ and $\brho(\mf g)(1)$ are non-isomorphic; For Item 4, we take $\sigma$ to be as above, the connected component of whose centralizer is denoted $T$, and take $\alpha$ to be a root of $\Phi(G,T)$. So (a) is satisfied. As $\brho(\mf g)$ is irreducible, (b) and (c) are satisfied.

Therefore, we can deform $\brho$ to a continuous representation
$\rho: \galQ \to \sl{2}{\mc O}$ satisfying the prescribed local conditions on $S$ and the Ramakrishna condition on a set of auxiliary primes disjoint from $S$. We write $G_{\rho}$ for the Zariski closure of the image of $\rho$ in $\sl{2}{\bQl}$. As $\brho(\mf g)$ is irreducible, $\op{Lie}(G_{\rho})$ is either trivial or $\mf g_{\bQl}$. If the former were true, then $G_{\rho}$ would be finite. But $\rho|_{\Gamma_p}$ is Steinberg, so in particular the image of $\rho$ is infinite, a contradiction. Thus $\op{Lie}(G_{\rho})=\mf g_{\bQl}$.
\end{proof}

Now we finish the proof of Theorem \ref{key cases} with the congruence conditions removed.
For $G$ a simple but non-simply connected group, suppose there is a homomorphism $\rho_l:\galQ \to G^{sc}(\bQl)$ with Zariski-dense image. We compose $\rho_l$ with the covering projection $G^{sc}(\bQl) \rsurj G(\bQl)$, the resulting map has Zariski-dense image in $G(\bQl)$.
Proposition \ref{weyl mono}, \ref{pr gl2 mono} and \ref{sl2 A5} prove the cases of a simple, simply-connected classical group, $E_6$ and $E_7$. On the other hand, the remaining cases $G_2, F_4$, and $E_8$ have already been established in \cite{pat16} in a way similar to the proof of Proposition \ref{pr gl2 mono}. So Theorem \ref{key cases} is proved. 

In order to remove the congruence conditions for $G$ of type $B_n, C_n$ and $E_7$, we impose the same local deformation conditions as specified in the paragraph preceding Lemma \ref{Wiles g}, but then use Theorem \ref{fkp,main} instead of Theorem \ref{lifting Weyl}. By Lemma \ref{Wiles g} and Lemma \ref{mu_p2}, the assumptions in Theorem \ref{fkp,main} are all met. Therefore, we obtain a characteristic zero lift of $\brho$ satisfying the prescribed local conditions for \emph{all} large enough primes $l$. Then the proof of Proposition \ref{weyl mono} shows that the lift has full monodromy group for all large enough primes.

\section{Connected reductive groups as monodromy groups}

Following \cite{mil07}, a connected algebraic group $G$ is said to be an \emph{almost-direct
product} of its algebraic subgroups $G_1, . . . , G_n$ if the map
$$G_1 \times \cdots \times G_n \to G: (g_1,\cdots, g_n) \mapsto g_1\cdots g_n$$
is a surjective homomorphism with finite kernel; in particular, this means that the $G_i$'s commute with each other and each $G_i$ is normal in $G$.

The following proposition is \cite[Corollary 4.4]{mil07}.
\begin{prop}\label{5.1}
An algebraic group is semi-simple if and only if it is an almost direct product of simple algebraic groups. (Here a simple algebraic group is called almost simple in \cite{mil07})
\end{prop}

\begin{prop}\label{5.2}(Goursat's Lemma)
Let $G_1, G_2$ be groups, let $H$ be a subgroup of $G_1 \times G_2$ such that the two projections $p_1: H \to G_1$, $p_2: H \to G_2$ are surjective. Let $N_1$ (resp. $N_2$) be the kernel of $p_2$ (resp. $p_1$). Then the image of $H$ in $G_1/N_1 \times G_2/N_2$ is the graph of an isomorphism $G_1/N_1 \cong G_2/N_2$. 
\end{prop}

\begin{prop}\label{5.3}
Let $G$ be a connected semi-simple group, then there are continuous homomorphisms 
$$\galQ \to G(\bQl)$$ with Zariski-dense image for large enough $l$.
\end{prop}
\begin{proof}
This will follow from Goursat's Lemma, Theorem \ref{key cases} and Remark \ref{remove cong}. By Proposition \ref{5.1}, it suffices to prove the case when $G$ is the direct product of simply-connected simple algebraic groups. We may decompose $G$ into `isotypic factors': $$G=G_1\times \cdots \times G_n$$
where $G_i$ is the direct product of copies of some simple algebraic group, and $G_i$, $G_j$ have different types for $i \neq j$. Suppose we are given $\rho_i: \galQ \to G_i(\bQl)$ with Zariski-dense image for each $i$ and let $\rho:=(\rho_1,\cdots,\rho_n): \galQ \to G(\bQl)$, whose Zariski-closure is denoted by $H$. Then $H$ is an algebraic subgroup of $G(\bQl)$ for which $\op{pr}_i(H)=G_i(\bQl)$. Since $G_i(\bQl)$ and $G_j(\bQl)$ share no common nontrivial quotients for $i \neq j$, Proposition \ref{5.2} implies that $H=G(\bQl)$.

It remains to prove the case when $G$ is a direct product of copies of some simply-connected simple algebraic group. Write $G=K^n$ with $K$ simple. We first assume $K \neq \op{SL}_2$. By 4.2.1 and 4.2.2 (especially Remarks \ref{non-conj 1} and \ref{non-conj 2}), there exists a prime $p$ and a homomorphism $\rho_i: \galQ \to K(\bQl)$ for $1 \leq i \leq n$ such that $\rho_i$ has Zariski-dense image and is unramified at $p$ with $\rho_i(\frob{p})$ a regular semisimple element in $K(\bQl)$, and for $i \neq j$, the images of $\rho_i(\frob{p})$ and $\rho_j(\frob{p})$ in $K^{ad}(\bQl)$ are not conjugate by an automorphism of $K^{ad}$. Now we use Proposition \ref{5.2} and induction on $n$ to show that $\rho:=\prod_i \rho_i: \galQ \to G(\bQl)$ has Zariski-dense image. This is clear when $n=1$. Suppose this is true for $n-1$, so that $\prod_{i<n}\rho_i: \galQ \to K^{n-1}(\bQl)$ has Zariski-dense image. Let $H$ be the Zariski-closure of the image of $\rho$ and apply Proposition \ref{5.2} to $p_1: H \to G_1=K^{n-1}$, $p_2: H \to G_2=K$, which are surjective by assumption, we see that the image of $H$ in $G_1/N_1 \times G_2/N_2$ is the graph of an isomorphism $G_1/N_1 \cong G_2/N_2$. But since $G_2=K$ is a simple, simply-connected algebraic group, $N_2=K$ or $N_2 \subset Z(K)$. Also note that 
$N_1$ is either $K^{n-1}$ or isogenous to the product of $n-2$ factors in $K^{n-1}$. If $N_2=K$ and $N_1=K^{n-1}$, $H$ must be $G$; otherwise, $N_2 \subset Z(K)$ and $N_1$ is isogenous to the product of $n-2$ factors in $K^{n-1}$, so the isomorphism $G_1/N_1 \times G_2/N_2$ induces an isomorphism between two factors in $G^{ad}=(K^{ad})^n$. But this is impossible, as for any $i \neq j$, the images of $\rho_i(\frob{p})$ and $\rho_j(\frob{p})$ in $K^{ad}(\bQl)$ are not conjugate by an automorphism of $K^{ad}$. Therefore, $H=G$.

Now let $K=\op{SL}_2$. By the argument in the previous paragraph, it suffices to show that for any $n$, we can construct homomorphisms
$\rho_i: \galQ \to \sl{2}{\bQl}$, $1 \leq i \leq n$ such that $\rho_i$ has Zariski-dense image and for $i \neq j$, $\rho_i^{ad}$ and $\rho_j^{ad}$ are not conjugated by an automorphism of $\op{PGL}_2$. 
By the construction in \cite{ser:tgt}, 9.3, there are infinitely many homomorphisms $r: \galQ \rsurj \tilde{A}_5$ such that the composites $\galQ \xrightarrow{r} \tilde{A}_5 \to A_5$ are ramified at different sets of finite primes. By 4.2.3, we then obtain infinitely many homomorphisms $\brho: \galQ \to \sl{2}{k}$ such that the corresponding homomorphisms $\brho^{ad}: \galQ \to \op{PGL}_2(k)$ are ramified at different sets of finite primes. By Proposition \ref{sl2 A5}, we can deform $\brho$ to characteristic zero with Zariski-dense image. We take $n$ of them, denoted by $\rho_1, \cdots, \rho_n: \galQ \to \sl{2}{\bQl}$. Suppose for some $i,j$ with $1 \leq i \neq j \leq n$, $\rho_i^{ad}$ and $\rho_j^{ad}$ were conjugate by an automorphism of $\op{PGL}_2$, then their mod $l$ reductions $\bar{\rho}_i^{ad}$ and $\bar{\rho}_j^{ad}$ would be conjugate as well. But since there exists a prime $p$ for which $\bar{\rho}_i^{ad}$ is ramified at $p$ and $\bar{\rho}_j^{ad}$ is unramified at $p$, $\bar{\rho}_i^{ad}$ and $\bar{\rho}_j^{ad}$ cannot be conjugate, a contradiction. 
\end{proof}

\begin{lem}\label{5.4}
Let $n$ be a positive integer and $T=(\mb G_m)^n$. Then there is a continuous map $\iota: \Zl \to T(\Ql)$ with Zariski-dense image.
\end{lem}
\begin{proof}
The group $T$ has only countably many connected proper Zariski-closed subgroups, so one can pick a line $L$ in $\op{Lie}(T)_{\Ql}$ avoiding the tangent spaces to all such proper subgroups (since $\Ql$ is uncountable). A small compact neighborhood of $0$ in $L$ exponentiates to a compact subgroup $C$ of $T(\Ql)$ whose Zariski closure has identity component that cannot be a proper algebraic subgroup of $T$, so $C$ is Zariski-dense in $T$.
\end{proof}

Now we can prove Theorem \ref{main}: Let $G$ be a connected reductive group, then
$G$ is a quotient of the product of $G^{der}$ (a semisimple group) and $Z(G)^0$ (a torus).
Proposition \ref{5.3} and Lemma \ref{5.4} then allow us to build a homomorphism from $\galQ$ to $G(\bQl)$ with Zariski-dense image for large enough $l$.

\end{document}